\documentclass[11pt]{amsproc}

\usepackage{amssymb}
\usepackage{graphicx}

\usepackage{tikz-cd}

\usepackage[all,knot]{xy}

\newcommand{\K}{\mathbb{K}}

\newcommand{\ot}{\otimes}

\DeclareMathOperator{\HH}{HH}
\DeclareMathOperator{\HHom}{Hom}
\DeclareMathOperator{\Img}{Im}
\DeclareMathOperator{\Ker}{Ker}

\newtheorem{theorem}{Theorem}[section]
\newtheorem{lemma}[theorem]{Lemma}

\theoremstyle{definition}
\newtheorem{definition}[theorem]{Definition}
\newtheorem{example}[theorem]{Example}

\theoremstyle{remark}
\newtheorem{remark}[theorem]{Remark}

\numberwithin{equation}{section}

\title[ Recurrence relations \& homotopy lifting maps]{Homotopy lifting maps on Hochschild cohomology and connections to deformation of algebras using reduction systems}

\author[T.\ Oke]{Tolulope Oke}
\address{Department of Mathematics \& Statistics, Wake Forest University, NC, USA}
\email{oket@wfu.edu}
\thanks{The author thanks Severin Barmeier for useful discussions and for assisting with the calculations in Section \ref{deform-g}.}
\subjclass[2020]{Primary 16E40; 16S37; 16S80, 16S15 }
\date{\today}

\begin{document}

\maketitle

\begin{abstract}
We describe the Gerstenhaber bracket structure on Hochschild cohomology of Koszul quiver algebras in terms of homotopy lifting maps. There is a projective bimodule resolution of Koszul quiver algebras that admits a comultiplicative structure. Introducing new scalars, we describe homotopy lifting maps associated to Hochschild cocycles using the comultiplicative structure. We show that the scalars can be described by some recurrence relations and we give several examples where these scalars appear in the literature. In particular, for a member of a family of quiver algebras, we describe Hochschild 2-cocycles and their associated homotopy lifting maps and determine the Maurer-Cartan elements of the quiver algebra in two ways: (i) by the use of homotopy lifting maps and (ii) by the use of a combinatorial star product that arises from the deformation of algebras using reduction systems.

\end{abstract}

\section{Introduction}\label{introduction}

The Hochschild cohomology $\HH^*(\Lambda)$ of an associative algebra $\Lambda$ possesses a multiplicative map called the cup product making it into a graded commutative ring. The ring structure of Hochschild cohomology of certain path algebras were determined using quiver techniques. For instance, if a path algebra is Koszul, its resolution possesses a comultiplicative structure and the cup product structure on its Hochschild cohomology can be presented using this comultiplicative structure. This cup product was described in \cite{ROKA}.

In addition to the cup product on Hochschild cohomology ring is the Gerstenhaber bracket making $\HH^*(\Lambda)$ into a graded Lie/Gerstenhaber algebra. The bracket plays an important role in the theory of deformation of algebras. The theory of deformation of algebras employs techniques in algebraic and noncommutative geometry to describe variations of the associative multiplicative structure on any algebra. In a recent article, S. Barmeier and Z. Wang \cite{SZ} have introduced a technique for finding families of deformation of quiver algebras with relations using \textit{reduction} systems. Reduction systems were introduced by Bergman \cite{Berg} in the late seventies. In particular, for an algebra $\Lambda=kQ/I$ where $Q$ is a finite quiver, there is associated a reduction system $R$ useful in determining a projective bimodule resolution of $\Lambda$ reminiscent of the Bardzell resolution \cite{CS}. It was shown in \cite{SZ} that there is an equivalence of formal deformations between (i) deformations of the associative algebra $\Lambda$, (ii) deformations of the reduction system $R$ and (iii) deformations of the relations in $I$.

The Gerstenhaber bracket can be difficult to compute in general settings. Several works have been carried out to interpret the bracket as well as make computations of the bracket accessible for a large class of algebras for instance in \cite{NW, KMOOW}. In \cite{VOL}, Y. Volkov introduced a method in which the bracket is defined in terms of a homotopy lifting map. This method works for any arbitrary projective bimodule resolution of the algebra. In earlier works \cite{TNO2}, we present a general formula for homotopy lifting maps associated to cocycles on Hochschild cohomology of Koszul path algebras.

The resolution introduced in \cite{ROKA} had scalars $c_{p,j}(n,i,r)$ appearing in the definition of the differentials on the resolution. These scalars made it possible to give a closed formula for the cup product structure on Hochschild cohomology. In Section \ref{main-results}, we present new scalars $b_{m,r}(m-n+1,s)$ associated to homotopy lifting maps on Hochschild cocycles using the scalars $c_{p,j}(n,i,r)$ of the comultiplicative relations. We show that the scalars $b_{m,r}(m-n+1,s)$ can be described using some recurrence relations and present the Gerstenhaber bracket structure using these scalars. We give several examples where these scalars appear in the literature in Section \ref{example-section}. In Section \ref{finding-mc}, we introduce a family of quiver algebras that has been extensively studied in \cite{TNO, TNO2}. For the algebra $A_1$  from the family, we find Hochschild 2-cocycles and their associated homotopy lifting maps. We show that $\HH^2(A_1)$ is generated as a vector space by five Maurer-Cartan elements.

Relevant results about Hochschild cohomology, Gerstenhaber bracket, quiver algebras and deformation of algebras using reduction systems were recalled in the preliminaries. The deformation of an algebra involves altering the associative multiplicative structure on the algebra. The candidates for determining how the multiplicative structure is altered can be described using a combinatorial star product $(\star)$\cite{SZ} and this product can be used to describe Maurer-Cartan elements. In Section \ref{deform-g}, we describe the Maurer-Cartan elements of the algebra $A_1$  using $(\star)$, showing that $\HH^2(A_1)$ is a five dimensional $k$-vector space.

\section{Preliminaries}\label{prelim}
\textbf{The Hochschild cohomology} of an associative $k$-algebra $\Lambda$ was originally defined using the following projective resolution known as the bar resolution.
\begin{equation}\label{bar}
\mathbb{B}_{\bullet}:  \qquad\cdots \rightarrow \Lambda^{\ot (n+2)}\xrightarrow{\;\delta_n\; }\Lambda^{\ot (n+1)}\xrightarrow{\delta_{n-1}} \cdots \xrightarrow{\;\delta_2\;}\Lambda^{\ot 3}\xrightarrow{\;\delta_1\;}\Lambda^{\ot 2}\;(\;\xrightarrow{\mu} \Lambda)
\end{equation}
where $\mu$ is multiplication and the differentials $\delta_n$ are given by
\begin{equation}\label{bar-diff}
\delta_n(a_0\otimes a_1\otimes\cdots\otimes a_{n+1}) = \sum_{i=0}^n (-1)^i a_0\otimes\cdots\otimes a_ia_{i+1}\otimes\cdots\otimes a_{n+1}
\end{equation}
for all $a_0, a_1,\ldots, a_{n+1}\in \Lambda$. This resolution consists of $\Lambda$-bimodules or left modules over the enveloping algebra $\Lambda^e = \Lambda\ot\Lambda^{op}$, where $\Lambda^{op}$ is the opposite algebra. The resolution is sometimes written $\mathbb{B}_{\bullet}\xrightarrow{\mu}\Lambda$ with $\mu$ referred to as the augmentation map. Let $M$ be a finitely generated left $\Lambda^e$-module, the Hochschild cohomology of $\Lambda$ with coefficients in $M$ denoted $ \HH^*(\Lambda,M)$ is obtained by applying the functor $\HHom_{\Lambda^e}(-,M)$ to the complex $\mathbb{B}_{\bullet}$, and then taking the cohomology of the resulting cochain complex. That is 
$$ \HH^*(\Lambda,M) := \bigoplus_{n\geq 0} \HH^n(\Lambda,M)=\bigoplus_{n\geq 0} \text{H}^n(\HHom_{\Lambda^e}(\mathbb{B}_{n},M)). $$
If we let $M=\Lambda$, we then define $\HH^*(\Lambda) :=\HH^*(\Lambda,\Lambda) $ to be the Hochschild cohomology of $\Lambda$. An element $\chi\in\HHom_{\Lambda^e}(\mathbb{B}_m,\Lambda)$ is a cocycle if $(\delta^{*}(\chi))(\cdot) := \chi\delta(\cdot)=  0.$ There is an isomorphism of the abelian groups $\HHom_{\Lambda^e}(\mathbb{B}_m,\Lambda)\cong\HHom_{k}(\Lambda^{\otimes m},\Lambda),$ so we can also view $\chi$ as an element of $\HHom_{k}(\Lambda^{\otimes m},\Lambda)$.

\textbf{The Gerstenhaber bracket} of two cocycles $\chi\in\HHom_{k}(\Lambda^{\otimes m},\Lambda)$ and $\theta\in \HHom_{k}(\Lambda^{\otimes n},\Lambda)$ at the chain level is given by
\begin{equation}\label{gbrac}
[\chi,\theta] = \chi\circ \theta - (-1)^{(m-1)(n-1)} \theta\circ \chi
\end{equation}
where $\chi\circ \theta = \sum_{j=1}^m (-1)^{(n-1)(j-1)}\chi\circ_j \theta$ with
\begin{multline*}
(\chi\circ_j \theta)(a_1\otimes \cdots \otimes a_{m+n-1}) \\
= \chi(a_1\otimes\cdots\otimes a_{j-1}\otimes \theta(a_j\otimes\cdots\otimes a_{j+n-1}) \otimes a_{j+n}\otimes \cdots\otimes a_{m+n-1}). \end{multline*}
This induces a well defined map  $[\cdot\;,\cdot] : \HH^{m}(\Lambda) \times \HH^{n}(\Lambda) \rightarrow \HH^{m+n-1}(\Lambda)$ on cohomology.

\textbf{Gerstenhaber bracket using homotopy lifting:} We present an equivalent definition of the Gerstenhaber bracket presented by Y. Volkov in \cite{VOL} and reformulated with a sign change by S. Witherspoon in Theorem \eqref{gbrac1}. We assume that $A$ is an algebra over the field $k$ and take $\mathbb{P}\xrightarrow{\mu_{\mathbb{P}}}A$ to be a projective resolution of $A$ as an $A^e$-module with differential $d^{\mathbb{P}}$ and augmentation map $\mu_{\mathbb{P}}.$ We take $\textbf{d}$ to be the differential on the Hom complex $\HHom_{\Lambda^e}(\mathbb{P},\mathbb{P})$ defined for any degree $n$ map $g:\mathbb{P}\rightarrow\mathbb{P}[-n]$ as 
$$\textbf{d}(g):= d^{\mathbb{P}}g - (-1)^ng d^{\mathbb{P}}$$
where $\mathbb{P}[-n]$ is a shift in homological dimension with $(\mathbb{P}[-n])_m = \mathbb{P}_{m-n}$.
In the following definition, the notation $\sim$ is used for two cocycles that are cohomologous, that is, they differ by a coboundary. 
\begin{definition}\label{homolift}
Let $\Delta_{\mathbb{P}}:\mathbb{P}\rightarrow \mathbb{P}\otimes_{A}\mathbb{P}$ be a chain map lifting the identity map on $A\cong A\ot_{A}A$ and suppose that $\eta\in\HHom_{A^e}(\mathbb{P}_n, A)$ is a cocycle. A module homomorphism $\psi_\eta :\mathbb{P} \rightarrow \mathbb{P}[1- n]$ is called a \textbf{homotopy lifting} map of $\eta$ with respect to $\Delta_{\mathbb{P}}$  if
\begin{align}\label{homo-defi}
\textbf{d}(\psi_\eta) &= (\eta\ot 1_{\mathbb{P}} - 1_{\mathbb{P}}\ot \eta)\Delta_{\mathbb{P}} \qquad\text{ and } \\ 
\mu_{\mathbb{P}}\psi_\eta &\sim \; (-1)^{n-1} \eta \psi  \notag
\end{align}
for some $\psi:\mathbb{P}\rightarrow \mathbb{P}[1]$ for which $\textbf{d}(\psi) =  (\mu_{\mathbb{P}}\ot 1_{\mathbb{P}} - 1_{\mathbb{P}}\ot \mu_{\mathbb{P}})\Delta_{\mathbb{P}}$. 
\end{definition}

\begin{example}\label{example-homo}
Let us consider a homotopy lifting formula for a cocycle $\beta$ using the bar resolution $\mathbb{B}$. Suppose that $\beta\in\HHom_{\Lambda^e}(\mathbb{B}_{n},A)\cong\HHom_{k}(A^{\ot n},A)$, then one way to define a homotopy lifting map $\psi_\beta:\mathbb{B}\rightarrow\mathbb{B}[1-n]$ for the cocycle $\beta$ is the following:
\begin{multline*}
\psi_\beta(1\ot a_1\ot\cdots\ot a_{m+n-1}\ot 1) \\
= \sum_{i=1}^{m}(-1)^{(m-1)(i-1)}1\ot a_1\ot\cdots\ot a_{i-1}\ot 
g(a_{i}\ot\cdots\ot a_{i+n-1})
\ot a_{i+n}\ot\cdots\ot a_{m+n-1}\ot 1.
\end{multline*}

We compute an example in which $\beta\in\HHom_{k}(\Lambda^{\ot 2},\Lambda)$. In degree $3$, $\psi_{\beta}:\mathbb{B}_{3}\xrightarrow{}\mathbb{B}_{2}$. Using the differentials on the bar resolution given in Equation \eqref{bar-diff} and the diagonal map $\Delta_{\mathbb{B}}$ later given in Equation \eqref{diag-bar}, we have
\begin{align*}
&\delta\psi_\beta(1\ot a_1\ot a_2\ot a_3\ot 1) = \beta(a_1\ot a_2)\ot a_3\ot 1 -1\ot \beta(a_1\ot a_2) a_3\ot 1 \\
&+ 1\ot \beta(a_1\ot a_2)\ot a_3 - a_1\ot \beta(a_2\ot a_3)\ot 1 \\
&+ 1\ot a_1 \beta( a_2\ot a_3)\ot 1  - 1\ot a_1\ot \beta(a_2\ot a_3) \qquad\text{and}\\
&\psi_\beta\delta(1\ot a_1\ot a_2\ot a_3\ot 1) = a_1\ot \beta(a_2\ot a_3)\ot 1 - 1\ot \beta(a_1 a_2\ot a_3)\ot 1 \\
&+1\ot \beta(a_1\ot a_2 a_3)\ot 1 - 1\ot \beta(a_1\ot a_2)\ot a_3.
\end{align*}
Therefore $(\delta\psi_\beta + \psi_\beta\delta) (1\ot a_1\ot a_2\ot a_3\ot 1) =\beta(a_1\ot a_2)\ot a_3\ot 1 - 1\ot a_1\ot \beta(a_2\ot a_3)$. On the other hand, 
\begin{align*}
&(\beta\ot 1 - 1\ot\beta)\Delta_{\mathbb{B}}(1\ot a_1\ot a_2\ot a_3\ot 1) \\
&= (\beta\ot 1 - 1\ot\beta)\Big( (1\ot 1)\ot_{\Lambda}(1\ot a_1\ot a_2\ot a_3\ot 1) + \\
&(1\ot a_1\ot1)\ot_{\Lambda}(1\ot a_2\ot a_3\ot 1) + (1\ot a_1\ot a_2\ot 1)\ot_{\Lambda}(1\ot a_3\ot 1) \\
&+ (1\ot a_1\ot a_2\ot a_3\ot 1)\ot_{\Lambda}(1\ot 1) \Big)\\
&=\beta(a_1\ot a_2)\ot a_3\ot 1 - 1\ot a_1\ot \beta(a_2\ot a_3).
\end{align*}
So we see that Equation \eqref{homo-defi} holds in degree 3 i.e.
$$\delta\psi_\beta - (-1)^{2-1} \psi_\beta \delta = (\beta\otimes 1 - 1\otimes \beta)\Delta_{\mathbb{B}}.$$
\end{example}

\begin{remark}\label{rema-koszul-homo}
Suppose that $\K$ is the Koszul resolution, then it is a differential graded coalgebra i.e. $(\Delta_{\K}\ot 1_{\K})\Delta_{\K} = (1_{\K}\ot \Delta_{\K})\Delta_{\K}$ and $(d\ot 1 + 1\ot d)\Delta_{\K} = \Delta_{\K}d$. Furthermore, the augmentation map $\mu:\K\rightarrow\Lambda$ makes $(\mu\ot 1_{\K})\Delta_{\K} - (1_{\K}\ot \mu)\Delta_{\K} = 0.$ We can therefore set $\psi = 0$ in the second part of Equation \eqref{homo-defi}, so that we have $\mu\psi_\eta \sim 0.$ Next, we set $\psi_\eta(\K_{n-1})=0$ and the second relations of Equation~\eqref{homo-defi} is satisfied. To check if a map is a homotopy lifting map, it is sufficient to verify the first equation in \eqref{homo-defi} if the resolution is Koszul.
\end{remark}
The following is a theorem of Y. Volkov which is equivalent to the definition of the bracket presented earlier in Equation \eqref{gbrac}.
\begin{theorem}\cite[Theorem 4]{VOL}\label{gbrac1}
Let $(\mathbb{P},\mu_{\mathbb{P}})$ be a $A^e$-projective resolution of the algebra $A$, and let $\Delta_{\mathbb{P}}:\mathbb{P}\xrightarrow{} \mathbb{P}\ot_{A}\mathbb{P}$ be a diagonal map. Let $\eta:\mathbb{P}_n\xrightarrow{}A$ and $\theta:\mathbb{P}_m\xrightarrow{}A$ be cocycles representing two classes. Suppose that $\psi_\eta$ and $\psi_\theta$ are homotopy liftings for $\eta$ and $\theta$ respectively. Then the Gerstenhaber bracket of the classes of $\eta$ and $\theta$ can be represented by the class of the element 
\begin{equation*}
[\eta, \theta]_{\Delta_{\mathbb{P}}} = \eta\psi_\theta - (-1)^{(m-1)(n-1)}\theta\psi_\eta.
\end{equation*}
\end{theorem}

\textbf{Quiver algebras:} A quiver is a directed graph with the allowance of loops and multiple arrows. A quiver $Q$ is sometimes denoted as a quadruple $(Q_0,Q_1,o,t)$ where $Q_0$ is the set of vertices in $Q$, $Q_1$ is the set of arrows in $Q$, and $o,t: Q_1 \longrightarrow Q_0$ are maps which assign to each arrow $a\in Q_1$, its origin vertex $o(a)$ and terminal vertex $t(a)$ in $Q_0$. A path in $Q$ is a sequence of arrows $a = a_1 a_2 \cdots a_{n-1} a_n $ such that the terminal vertex of $a_{i}$ is the same as the origin vertex of $a_{i+1}$, using the convention of concatenating paths from left to right. The quiver algebra or path algebra $kQ$ is defined as a vector space having all paths in $Q$ as a basis. Vertices are regarded as paths of length $0$, an arrow is a path of length $1$, and so on. We take multiplication on $kQ$ as concatenation of paths. Two paths $a$ and $b$ satisfy $ab=0$ if $t(a)\neq o(b)$. This multiplication defines an associative algebra over $k$. 
By taking $kQ_i$ to be the $k$-vector subspace of $kQ$ with paths of length $i$ as basis, $kQ = \bigoplus_{i\geq 0} kQ_i$ can be viewed as an $\mathbb{N}$-graded vector space. Two paths are parallel if they have the same origin and terminal vertex. A relation on a quiver $Q$ is a linear combination of parallel paths in $Q$. A quiver together with a set of relations is called a quiver with relations. Letting $I$ be an ideal of the path algebra $kQ$, we denote by $(Q,I)$ the quiver $Q$ with relations $I$. The quotient $\Lambda  = kQ/I$ is called the quiver algebra associated with $(Q,I)$. Suppose that $\Lambda$ is graded by positive integers and is Koszul, the degree $0$ component $\Lambda_0$ is isomorphic to $k$ or copies of $k$ and $\Lambda_0$ has a linear graded projective resolution $\mathbb{L}$ as a right $\Lambda$-module \cite{HCA,MV}.

An algorithmic approach to finding such a minimal projective resolution $\mathbb{L}$ of $\Lambda_0$ was given in~\cite{AAR}. The modules $\mathbb{L}_{n}$ are right $\Lambda$-modules for each $n$. There is a ``comultiplicative structure" on $\mathbb{L}$ and this structure was used to find a minimal projective resolution $\K\rightarrow\Lambda$ of modules over the enveloping algebra of $\Lambda$ in~\cite{ROKA}. A non-zero element $x\in kQ$ is called uniform if it is a linear combination of paths each having the same origin vertex and the same terminal vertex: In other words, $x=\sum_{j} c_j w_j$ with scalars $c_j\neq 0$ for all $j$ and each path $w_j$ are of equal length having the same origin vertex and the same terminal vertex. For $R=kQ$, it was shown in \cite{AAR} that there are integers $t_n$  and uniform elements $\{f^n_i\}_{i=0}^{t_n}$ such that the right projective resolution $\mathbb{L}\rightarrow\Lambda_0$ is obtained from a filtration of $R$. This filtration is given by the following nested family of right ideals: 
$$\cdots \subseteq \bigoplus_{i=0}^{t_n} f^n_iR \subseteq \bigoplus_{i=0}^{t_{n-1}} f^{n-1}_iR\subseteq \cdots \subseteq\bigoplus_{i=0}^{t_1} f^1_iR\subseteq \bigoplus_{i=0}^{t_0} f^0_iR = R$$
where for each $n$, $\mathbb{L}_n = \bigoplus_{i=0}^{t_n} f^n_iR / \bigoplus_{i=0}^{t_n} f^n_iI$ and the differentials on $\mathbb{L}$ are induced by the inclusions $\bigoplus_{i=0}^{t_n} f^n_iR\subseteq \bigoplus_{i=0}^{t_{n-1}} f^{n-1}_iR$. Furthermore, it was shown in \cite{AAR} that with some choice of scalars, the $\{f^n_i\}_{i=0}^{t_n}$ satisfying the comultiplicative equation of \eqref{comult-struc} make $\mathbb{L}$ minimal. In other words, for $0\leq i\leq t_n$, there are scalars $c_{pq}(n,i,r)$ such that
\begin{equation}\label{comult-struc}
f^n_i = \sum_{p=0}^{t_r}\sum_{q=0}^{t_{n-r}}c_{pq}(n,i,r) f^r_p f^{n-r}_q
\end{equation}
holds and $\mathbb{L}$ is a minimal resolution. To construct the above multiplicative equation for example, we can take $\{f^0_i\}_{i=0}^{t_0}$ to be the set of vertices, $\{f^1_i\}_{i=0}^{t_1}$ to be the set of arrows, $\{f^2_i\}_{i=0}^{t_2}$ to be the set of uniform relations generating the ideal $I$, and define $\{f^n_i\}_{i=0}^{t_n}(n\geq 3)$ recursively, that is in terms of $f^{n-1}_i$ and $f^1_j$. We presented the comultiplicative structure of a family of quiver algebras in \cite{TNO} and use the homotopy lifting technique to show that for some members of the family, the Hochschild cohomology ring modulo the weak Gerstenhaber ideal generated by homogeneous nilpotent elements is not finitely generated.

The resolution $\mathbb{L}$ and the comultiplicative structure~\eqref{comult-struc} were used to construct a minimal projective resolution $\K\rightarrow \Lambda$ of modules over the enveloping algebra $\Lambda^e=\Lambda\otimes \Lambda^{op}$ on which we now define Hochschild cohomology. This minimal projective resolution $\K$ of $\Lambda^e$-modules associated to $\Lambda$ was given in \cite{ROKA} and now restated with slight notational changes below.

\begin{theorem}\cite[Theorem 2.1]{ROKA}\label{mainres}
Let $\Lambda=kQ/I$ be a Koszul algebra, and let $\{f^n_i\}_{i=0}^{t_n}$ define a minimal resolution of $\Lambda_0$ as a right $\Lambda$-module. A minimal projective resolution $(\K, d)$ of $\Lambda$ over $\Lambda^e$ is given by
\begin{equation*}
\mathbb{K}_n = \bigoplus_{i=0}^{t_n}\Lambda o(f_i^n)\otimes_k t(f_i^n)\Lambda
\end{equation*} 
for $n\geq 0$, where the differential $d_n:\K_n\xrightarrow{}\K_{n-1}$ applied the basis element $\varepsilon^n_i = (0,\ldots,0,o(f^n_i)\otimes_k t(f^n_i),0,\ldots,0),$ $0\leq i\leq t_n$  with $o(f^n_i)\otimes_k t(f^n_i)$  in the $i$-th position, is given by
\begin{equation}\label{diff-k}
d_n(\varepsilon^n_i ) = \sum_{j=0}^{t_{n-1}}\Big( \sum_{p=0}^{t_1}c_{p,j}(n,i,1) f_p^1 \varepsilon^{n-1}_j 
+ (-1)^n \sum_{q=0}^{t_1}c_{j,q}(n,i,n-1)\varepsilon^{n-1}_j  f_q^1 \Big) 
\end{equation}
and $d_0:\K_0\xrightarrow{}\Lambda$ is the multiplication map. In particular, $\Lambda$ is a linear module over $\Lambda^e$.
\end{theorem}

We note that for each $n$ and $i$, $\{\varepsilon^n_i\}_{i=0}^{t_n}$ is a basis of $\K_n$ as a $\Lambda^e$-module. The scalars $c_{p,j}(n,i,r)$ are those appearing in~\eqref{comult-struc} and $f^1_{*} := \overline{f^1_{*}}$ is the residue class of $f^1_{*}$ in  $\bigoplus_{i=0}^{t_1} f^1_iR / \bigoplus_{i=0}^{t_n} f^1_i I$. Using the comultiplicative structure of Equation~\eqref{comult-struc}, a cup product formula on Hochschild cohomology of Koszul quiver algebra was presented in~\cite{MSKA} using the resolution $\K$.

We recall the definition of the reduced bar resolution of algebras defined by quivers and relations. If $\Lambda_0$ is isomorphic to $m$ copies of $k$, take $\{e_1,e_2,\ldots,e_m\}$ to be a complete set of primitive orthogonal central idempotents of $\Lambda$. In this case $\Lambda$ is not necessarily an algebra over $\Lambda_0$. If $\Lambda_0$ is isomorphic to $k$, then $\Lambda$ is an algebra over $\Lambda_0$. For convenience, we use the same notation $\mathbb{B}$ for both the bar resolution and the reduced bar resolution. The reduced bar resolution $(\mathbb{B}, \delta)$, where  $\mathbb{B}_n := \Lambda^{\ot_{\Lambda_0}(n+2)}$ is the $(n+2)$-fold tensor product of $\Lambda$ over $\Lambda_0$ and uses the same differential as the usual bar resolution presented in Equation~\eqref{bar-diff}. The resolution $\K$ can be embedded naturally into the reduced bar resolution $\mathbb{B}$. There is a map $\iota:\K\rightarrow\mathbb{B}$ defined by $\iota(\varepsilon^n_r) = 1\ot \widetilde{f^n_r}\ot 1$ such that $\delta\iota=\iota d$, where 
\begin{equation}\label{the-f}
\widetilde{f^n_j} = \sum c_{j_1j_2\cdots j_n} f^1_{j_1}\ot f^1_{j_2}\ot \cdots\ot f^1_{j_n}\quad\text{ if }\quad f^n_j = \sum c_{j_1j_2\cdots j_n} f^1_{j_1} f^1_{j_2} \cdots f^1_{j_n}
\end{equation}
for some scalar $c_{j_1j_2\cdots j_n}$. It was shown in \cite[Proposition 2.1]{MSKA} that $\iota$ is indeed an embedding. By taking $\Delta_{\mathbb{B}}:\mathbb{B}\rightarrow\mathbb{B}\ot_{\Lambda}\mathbb{B}$ to be the following comultiplicative map (or diagonal map) on the bar resolution,
\begin{equation}\label{diag-bar}
\Delta_{\mathbb{B}}(a_0\otimes\cdots\otimes a_{n+1}) = \sum_{i=0}^n (a_0\otimes\cdots\otimes a_i\ot 1)\ot_{\Lambda}( 1\ot a_{i+1}\otimes\cdots\otimes a_{n+1}),
\end{equation}
it was also shown in~\cite[Proposition 2.2]{MSKA} that the diagonal map $\Delta_{\K}:\K\rightarrow\K\ot_{\Lambda}\K$ on the complex $\K$ has the following form.
\begin{equation}\label{diag-k}
\Delta_{\K}(\varepsilon^n_r) = \sum_{v=0}^{n}\sum_{p=0}^{t_v}\sum_{q=0}^{t_{n-v}}c_{p,q}(n,r,v)\varepsilon^v_p\ot_{\Lambda}\varepsilon^{n-v}_q.
\end{equation}
The compatibility of $\Delta_{\K}, \Delta_{\mathbb{B}}$ and $\iota$ means that $(\iota\ot \iota)\Delta_{\K} = \Delta_{\mathbb{B}} \iota$ where $(\iota\ot \iota)(\K\ot_{\Lambda}\K) = \iota(\K)\ot_{\Lambda}\iota(\K)\subseteq \mathbb{B}\ot_{\Lambda}\mathbb{B}.$

\textbf{Deformation of algebras using reduction system:}
The theory of deformation of algebras has a much more wider scope than the contents of this article. There has been survey articles covering several aspects of algebraic deformation theory - from deformations arising from noncommutative geometry to formal, infinitesimal and graded deformations. References were made to such articles in \cite[chapter 5]{HCA}. Although many results are known for deformation of commutative algebras, little is known of deformation of quiver and path algebras. Let $\Lambda=kQ/I$ and $Q$ a finite quiver. There is associated to $\Lambda$, a reduction system $R=\{(s,\varphi_s)\}$ given by Definition \ref{red-sys} and the Gerstenhaber bracket equips Hochschild cohomology $\HH^*(\Lambda)$ with a DG Lie algebra structure which controls the theory of deformation of $R$. Using a combinatorial star product, it was shown in \cite{SZ} that the deformations of $\Lambda$ is equivalent to the deformations of the reduction system $R$ which is also equivalent to the deformation of the relations in $I$. There is a map $\varphi:kQ\rightarrow k Irr_S$ defined by $\varphi(s)=\varphi_s$ sending a path in the quiver algebra to an irreducible path. Let $\pi:kQ\rightarrow \Lambda$ be the projection map. The combinatorial star product on $\Lambda$ defined on irreducible paths $u,v\in kQ$ is the image of $\pi(u)\star\pi(v)$ in $\Lambda$ after performing right-most \textit{reductions} on the path $uv\in kQ$.

Suppose that $(\Lambda_\tau, \mu_\tau)$ is a formal deformation (Definition \ref{defi-formal-deformation}) of the associative multiplication on $\Lambda$, one way to describe the deformed multiplication $\mu_\tau$ is by obtaining a necessary and sufficient condition for the associativity of the combinatorial star product. This condition is given by Equation \eqref{comb-star}. There is a projective bimodule resolution $\mathfrak{p}(Q,R)$ arising from a reduction system $R$ and the combinatorial star product can be used to describe Maurer-Cartan elements in $\mathfrak{p}(Q,R)\otimes \mathfrak{m},$ where $(N,\mathfrak{m})$ is a complete local Noetherian $k$-algebra. In Section \ref{deform-g}, we use the combinatorial star product to determine Maurer-Cartan elements of $\HH^2(A_1)$, thus determining a family of deformations of the algebra $A_1$. In the future, it will be interesting to find meaningful ways to describe Maurer-Cartan elements that were obtained using the star product (as in Example \ref{final-example}) in terms of those obtained by the homotopy lifting technique (as in Example \ref{mc-elements}) and vice versa. We now give a result from \cite{Berg} on reduction systems.
\begin{definition}\label{red-sys}
Let $\Lambda = kQ/I$ be a path algebra with finite quiver $Q$. A reduction system $R$ for $kQ$ is a set of pairs
$$R = \{ (s,\varphi_s) \;|\; s\in S, \varphi_s\in kQ\}$$
where 
\begin{itemize}
\item $S$ is a subset of paths of length greater than or equal to 2 such that $s$ is not a subpath of $s'$ when $s\neq s'\in S$
\item $s$ and $\varphi_s$ are parallel paths
\item for each $(s,\varphi_s)\in R$, $\varphi_s$ is irreducible or a linear combination of irreducible paths.
\end{itemize}
\end{definition}

We say a path is \textit{irreducible} if it does not contain elements in $S$ as a subpath.
\begin{definition}\label{diamond-defi}
Given a two-sided ideal $I$ of $kQ$, we say that a reduction system $R$ satisfies the diamond condition $(\diamond)$ for $I$ if
\begin{itemize}
\item $I$ is equal to the two-sided ideal generated by the set $\{s-\varphi_s\}_{(s,\varphi_s)\in R}$ and
\item every path is \textit{reduction unique}.
\end{itemize}
We call a reduction system $R$ finite if $R$ is a finite set.
\end{definition}
\begin{definition}
Let $R$ be a reduction system for $kQ$ and $p,q,r$ be paths of length at least 1. A path $pqr$ of length at least 3 is an \textit{overlap ambiguity} of $R$ if $pq, qr\in S$. The set of all paths having one overlap is also the set of all 1-ambiguity and is denoted $S_3$.
\end{definition}
 We now state the diamond lemma as provided in \cite{Berg}.
\begin{theorem}\cite[Thm 1.2]{Berg}
Let $R=\{(s,\varphi_s)\}$ be a reduction system for $kQ$ and let $S=\{s \;|\;(s,\varphi_s)\in R\}$. Denote by $I=\langle s-\varphi_s\rangle_{s\in S}$ the corresponding two-sided ideal of $\Lambda=kQ/I$. If $R$ is reduction finite, the following are equivalent:
\begin{itemize}
\item all overlap ambiguities of $R$ are resolvable
\item $R$ is reduction unique, that is $R$ satisfies $(\diamond)$ for $I$
\item the image of the set of irreducible paths under the projection $\pi:kQ\rightarrow \Lambda$ forms a $k$-basis for $\Lambda$.
\end{itemize}
\end{theorem}

It is known that for any two-sided ideal $I$ of the quiver algebra $kQ/I$, there is always a choice of a reduction system $R$ satisfying the diamond condition $(\diamond)$ of Definition \ref{diamond-defi}. This is a result of S. Chouhy and A. Solotar \cite[Prop 2.7, Thm 4.1, Thm 4.2]{CS}. Furthermore, there is a projective bimodule resolution $\mathfrak{p}(Q,R)$ associated to the algebra $\Lambda=kQ/I$ useful in extracting information about Hochschild cohomology.

Let $B$ be a $k$-algebra and let $\tau$ be an indeterminate. The ring $B[[\tau]]$ is the ring of formal power series in $\tau$ with coefficients in $B$. The ring $B[[\tau]]$ is a $k[[\tau]]$-module if we identify $k$ with the subalgebra $k\cdot 1$ of $B$. The multiplication in $B$ is usually denoted by concatenation while the multiplication in $B[[\tau]]$ is given as 
$$(\sum_{i\geq 0} a_i\tau^i)(\sum_{j\geq 0} b_j\tau^j) = \sum_{m\geq 0}(\sum_{i+j=m} a_ib_j\tau^m)$$
We are interested in a new associative structure on $B$. All such associative structure provides a deformation $B_\tau$ of the algebra $B$.

\begin{definition}\label{defi-formal-deformation}
A \textit{formal} deformation $(B_\tau,\mu_\tau)$ of $B$ (also called a deformation of $A$ over $k[[\tau]]$) is an associative bilinear multiplication $\mu_\tau:B[[\tau]]\otimes B[[\tau]]\rightarrow B[[\tau]]$ such that in the quotient by the ideal $(\tau)$, the multiplication $\mu_\tau(b_1,b_2)$ coincides with the multiplication in $B$ for all $b_1,b_2\in B[[\tau]]$.
\end{definition}
The multiplication $\mu_\tau$ above is determined by products of pairs of elements of $B$, so that for every $a,b\in B$
\begin{equation}
\mu_\tau(a,b) = ab +\mu_1(a,b)\tau + \mu_2(a,b)\tau^2 + \mu_3(a,b)\tau^3+\cdots
\end{equation}
where $ab$ is the usual multiplication in $B$ and $\mu_i:B\otimes_k B\rightarrow B$ are bilinear maps. If we denote the usual multiplication in $B$ by $\mu$, we may denote a deformation of $(B,\mu)$ by $(B_\tau,\mu_\tau)$ where 
\begin{equation}\label{mu}
\mu_\tau = \mu + \mu_1\tau +\mu_2\tau^2 +\mu_3\tau^3 + \cdots
\end{equation}
\begin{remark}\
The \textit{first order term} i.e. the bilinear map $\mu_1$ is called an \textit{infinitesimal deformation} if it is a Hochschild 2-cocycle. Furthermore, if $\mu_1$ is an infinitesimal deformation, it defines a deformation of $B$ over $k[\tau]/(\tau^2)$ and satisfies
$$\mu_1(ab,c)+\mu_1(a,b)c = \mu_1(a,bc)+a\mu_1(b,c).$$ 
This equation is derived from the associative property that the bilinear map $\mu_\tau$.
\end{remark}
While formal deformations are deformations over $k[[\tau]]$, algebraic deformations are deformations over $k[\tau]$. The idea of making a formal deformation into an algebraic deformation is called the \textit{algebraization of formal deformations} and were examined in details with respect to reduction systems by S. Barmeier and Z. Wang in \cite{SZ}. One of their main results is the following theorem.
\begin{theorem}\cite[Thm 7.1]{SZ}\label{equiv-deform}
Given any finite quiver $Q$ and any two-sided ideal of relations $I$, let $\Lambda=kQ/I$ be the quotient algebra and let $R$ be any reduction system satisfying the diamond condition $(\diamond)$ for $I$. There is an equivalence of formal deformation problems between
\begin{enumerate}
\item[(i)] deformations of the associative algebra structure on $\Lambda$
\item[(ii)] deformations of the reduction system $R$
\item[(iii)] deformations of the relations $I$
\end{enumerate}
\end{theorem}
This equivalence can be made explicit by considering combinatorial star products which produce a deformation of $\Lambda$ from a deformation of the reduction system $R$.
In \cite[Section 5]{SZ}, it was established that there are comparison morphisms $F_{\bullet}, G_{\bullet}$ between the bar resolution $\mathbb{B}$ and the resolution $\mathfrak{p}(Q,R)$,
$$\mathfrak{p}_{\bullet}\xrightarrow{F_{\bullet}}\mathbb{B}_{\bullet} \xrightarrow{G_{\bullet}}  \mathfrak{p}_{\bullet}$$
coming from reduction systems and the combinatorial star product can be defined in terms of these morphisms.

Given a reduction system $R=\{(s,\varphi_s)\}$ for the algebra $B=kQ/I$ determined by $S$ as in Definition \ref{red-sys}, we view $\varphi\in\HHom(kS,k Irr_S)$ with $\varphi(s)=\varphi_s$. There is an isomorphism $B\cong kIrr_S$ so that $\HHom(kS, kIrr_S)\cong \HHom(kS,B)$ and a Hochschild 2-cochain in $\HHom(B\otimes_k B,B)$ may be viewed as an element $\varphi\in \HHom(kS,B)$. Taking $\mathfrak{m}=(\tau)$ as the maximal ideal of $k[[\tau]]$, it was shown that the map $\tilde{\varphi}\in\HHom(kS,B)\otimes\mathfrak{m}$ given by
\begin{equation}\label{varphi}
\widetilde{\varphi} = \widetilde{\varphi}_1t +\widetilde{\varphi}_2t^2 + \widetilde{\varphi}_3t^3+\cdots
\end{equation}
are candidates for deformations of $R$ determined by $\varphi$. Moreover, the deformation of $R$ given by $\varphi+\widetilde{\varphi}$ is a deformation of the algebra $B$ if and only if $\widetilde{\varphi}$ is a Maurer-Cartan element of the $L_\infty$ algebra structure on the resolution $\mathfrak{p}(Q,R)\otimes\mathfrak{m}$. More precisely, if $uvw\in S_3$ are overlaps such that $uv, vw\in S$, then $\widetilde{\varphi}$ satisfies the Maurer-Cartan equation if and only if
\begin{equation}\label{comb-star}
(\pi(u)\star\pi(v))\star\pi(w) = \pi(u)\star(\pi(v)\star\pi(w)).
\end{equation}
We recall that the product $\star$ defined on irreducible paths $u,v$ is the image of $\pi(u)\star\pi(v)$ in $\Lambda$ under the map $\pi:kQ\rightarrow \Lambda$ after performing right most reductions on the path $uv\in kQ$ using the reduction system. Indeed, the combinatorial star product $\star$ defines an associative structure on the algebra $(B_\tau,\mu_\tau)$ and we can also write $ a\star b = ab + \mu_1(a,b)$.

In Sections \ref{finding-mc}, we explicitly find Maurer-Cartan elements of $\HH^2(A_1)$ by first solving for Hochschild 2-cocycles $\mu_1$ of Equation \eqref{mu} and then showing that they satisfy the equation $d^*\mu+\frac{1}{2}[\mu,\mu]=0$ using homotopy lifting maps to define the bracket. The combinatorial star product solves the Maurer-Cartan equation by construction. We check in Section \ref{deform-g} that for the same algebra $A_1$ and a choice reduction system $R$, the maps $\tilde{\varphi}$ of Equation \eqref{varphi} describing the combinatorial star product solves the Maurer-Cartan equation given previously in terms of the homotopy lifting maps in Section \ref{finding-mc}.

\section{Main Results}\label{main-results}
In what follows, we consider a finite quiver $Q$ and a quiver algebra $\Lambda = kQ/I$ that is Koszul i.e. $I$ is an admissible ideal generated by paths of length 2. We assume the quiver $Q$ has arrows labelled $f^1_1, f^1_2,\ldots,f^1_{t_1}$ for some integer $t_1$. We suppose further that for each $n$, there are uniform elements $f^n_1, f^n_2,\ldots,f^n_{t_n}$, for some integer $t_n$ defining a minimal projective resolution $\K$ of $\Lambda$ as given by Theorem \eqref{mainres}. Let $\eta:\K_n\rightarrow\Lambda$ be a Hochschild cocycle such that for some index $i$, $\eta(\varepsilon^n_i)=f^1_w$ (resp. $\eta(\varepsilon^n_i)=f^1_wf^1_{w'}$) with $0\leq w,w'\leq t_1$ and $\eta(\varepsilon^n_j)=0$ for all $i\neq j$. We also write 
$\eta = \begin{pmatrix} 0 & \cdots & 0 & (f^1_w)^{(i)} & 0 & \cdots  & 0\end{pmatrix}$ for this type of cocycle (resp. $\eta = \begin{pmatrix} 0 & \cdots & 0 & (f^1_wf^1_{w'})^{(i)} & 0 & \cdots  & 0\end{pmatrix}$). Let $\Delta_{\K}:\K\rightarrow \K\otimes_{\Lambda}\K$ be the diagonal map. Results from \cite{VOL, HCA} establish that there exist maps $\psi_\eta:\K\rightarrow{}\K[1-n]$ such that
\begin{equation}\label{homo-defi-defi}
d\psi_\eta - (-1)^{1-n}\psi_\eta d = (\eta\ot 1 - 1 \ot \eta)\Delta_{\K} 
\end{equation}
for Koszul algebras. These maps are called \textit{homotopy lifting} maps for $\eta$. How would we define such a map explicitly in terms of the basis elements $\varepsilon^n_r$? Can we give a closed formula or expression of the Gerstenhaber bracket using an explicitly described versions of these maps? These are among the questions we address in this section.

In order to distinguish the index $n$ which is the degree of the cocycle $\eta$, we will index the resolution $\K$ by $m$ so that each $\K_m$ is free and generated by $\{\varepsilon^m_r\}_{r=0}^{t_m}$. For an $n$-cocycle, the map $\psi_\eta:\K_{\bullet}\rightarrow\K_{\bullet}$ associated to $\eta$ shifts is of degree $1-n$ so that for a fixed $m$, $\psi_\eta:\K_m\rightarrow{}\K_{m-n+1}$. Suppose that $\K_{m-n+1}$ is generated by $\{\varepsilon^{m-n+1}_{r'}\}_{r'=0}^{t_{m-n+1}}$, fundamental results from linear algebra means such a map is a $t_{m-n+1}\times t_{m}$ matrix when the modules are considered as left $\Lambda^e$-modules. In particular, for $\eta = \begin{pmatrix} 0 & \cdots & 0 & (f^1_w)^{(i)} & 0 & \cdots  & 0\end{pmatrix},$  such a map is defined on $\varepsilon^m_r$ by
\begin{equation}\label{map-gen}
\psi_\eta(\varepsilon^m_r) = \sum_{j=0}^{t_{m-n+1}} \lambda_{j}(m,r)\varepsilon^{m-n+1}_{j}
\end{equation} 
and for $\eta = \begin{pmatrix} 0 & \cdots & 0 & (f^1_wf^1_{w'})^{(i)} & 0 & \cdots  & 0\end{pmatrix}$ , such a map is defined on $\varepsilon^m_r$ by
\begin{equation}\label{map-gen-1}
\psi_\eta(\varepsilon^m_r) = \sum_{j=0}^{t_{m-n+1}} \lambda_{j}(m,r)f^1_w\varepsilon^{m-n+1}_{j}+ \lambda'_{j}(m,r)\varepsilon^{m-n+1}_{j}f^1_{w'}
\end{equation} 
where $\lambda_{j}(m,r),\lambda'_j(m,r)\in\Lambda^e$ in general. For details about these maps, see \cite{TNO2}. We now restrict Equation \eqref{map-gen} to the special case where for some $j=r'$, $\lambda_{j}(m,r)=b_{m,r}(m-n+1,r')$ is a scalar and $\lambda_{j}(m,r)=0$ for all $j\neq r'$, that is
\begin{equation}\label{map-special}
\psi_\eta(\varepsilon^m_r) = b_{m,r}(m-n+1,r')\varepsilon^{m-n+1}_{r'}
\end{equation}
and restrict Equation \eqref{map-gen-1} to the special case where all $\lambda_j(m,r), \lambda^{'}_j(m,r)$ are zero except for some indices $s$ and $s'$ with $\lambda_j(m,r) = \lambda_{m,r}(m-n+1,s)\neq 0$ and $\lambda^{'}_j(m,r) = \lambda_{m,r}(m-n+1,s')\neq 0$. That is,
\begin{equation}\label{map-special-2}
\psi_\eta(\varepsilon^m_r) = \lambda_{m,r}(m-n+1,s)f^1_w\varepsilon^{m-n+1}_{s}+ \lambda_{m,r}(m-n+1,s')\varepsilon^{m-n+1}_{s'}f^1_{w'}.
\end{equation}
It was shown in \cite{TNO2} that Equation \eqref{homo-defi-defi} holds under certain conditions on the scalars $b_{m,r}(m-n+1,r')$, and therefore the special maps given by Equations \eqref{map-special} are indeed homotopy lifting maps for the associated cocycles. A similar argument holds for the map given by Equation \eqref{map-special-2}.

Our motivation for defining the maps the way they were defined comes from several examples that were computed. We note in particular that the scalars $b_{m,r}(m-n+1,r')$ satisfy some recurrence relations given by Equation \eqref{big-equation}. We observe that if $\psi_\eta(\varepsilon^{m-1}_{\bar{r}}) = b_{m-1,\bar{r}}(m-n,r'')\varepsilon^{m-n}_{r''}$, we can obtain the scalars $b_{m,r}(m-n+1,r')$ in terms of the scalars $b_{m-1,\bar{r}}(m-n,r'')$ and the scalars $c_{pq}(n,i,r)$ coming from the comultiplicative structure on $\K$. The following diagram is not commutative but gives a pictures of this idea:
$$ \begin{tikzcd}
\K: =\arrow{d}{\psi_{\bar{\eta}}} \cdots \arrow{r} 
    & \K_{m+1}\arrow{d}{\psi_{\bar{\eta}}}\arrow{r}
        & \K_m \arrow{d}{\psi_{\bar{\eta}}} \arrow{r}
            & \K_{m-1} \arrow{d}{\psi_{\bar{\eta}}} \arrow{r}
                & \cdots \\
\K[1-n]:= \cdots \arrow{r}
    & \K_{m-n+2} \arrow{r}
        & \K_{m-n+1} \arrow{r}
            & \K_{m-n} \arrow{r}
                & \cdots
\end{tikzcd}$$
$\bullet$ We can obtain $\psi_{\bar{\eta}}\vert_{\K_{m+1}}$ from $\psi_{\bar{\eta}}\vert_{\K_{m}}$ for every $m$ using the scalars  $b_{m,r}(m-n+1,r^*)$ and $c_{pq}(n,i,r)$. From Remark \ref{rema-koszul-homo}, the scalars $b_{n-1,r^*}(0,r^{**})=0$ for all $r^*, r^{**}$ since $\psi_{\bar{\eta}}\vert_{\K_{n-1}}$ is the zero map.

For a finite quiver $Q$, let $\Lambda=kQ/I$ be a quiver algebra that is Koszul and let $\K_m$ be the projective bimodule resolution of $\Lambda$ with basis $\{\varepsilon^m_r\}_{r=0}^{t_m}$ as given by Theorem \ref{mainres}. For the specific modules $\K_{m-n+1}$,$\K_{m-n}$, and $\K_{m-1}$, let the basis elements be
$\{\varepsilon^{m-n+1}_{r'}\}_{r'=0}^{t_{m-n+1}}$, $\{\varepsilon^{m-n}_{r''}\}_{r''=0}^{t_{m-n}}$, and $\{\varepsilon^{m-1}_{\bar{r}}\}_{\bar{r}=0}^{t_{m-1}}$ respectively. For the indices $r,\bar{r},r'',r'$, let the scalars $b_{m,r}(*,**)$ be such that the following recurrence relations hold;
\begin{align}\label{big-equation}
&(i)\; b_{m,r}(m-n+1,r') c_{rr''}(m-n+1,r',1)  \\
\notag &= (-1)^{m-1}b_{m-1,\bar{r}}(m-n,r'') c_{r\bar{r}}(m,r,1) +  c_{ir''}(m,r,n),\\
\notag &b_{m,r}(m-n+1,r') c_{r''r}(m-n+1,r',m-n) \\
\notag &= (-1)^{m-1} b_{m-1,\bar{r}}(m-n,r'') c_{\bar{r}r}(m,r,m-1) -(-1)^{n(m-n)}c_{r''i}(m,r,m-n),\\
\notag &\text{  and for every pair of indices  } (p,q)\neq (r,r'), (p,q)\neq (r'',r), \\
\notag &(ii)\; b_{m,r}(m-n+1,r') c_{pq}(m-n+1,r',*) \\
\notag & = (-1)^{m-1} b_{m-1,\bar{*}}(m-n,*'') c_{pq}(m,r,*).
\end{align}

We start with the following Lemma. 

\begin{lemma}\label{lemma-homotopylifting-1}
Let $Q$ be a finite quiver and $\Lambda=kQ/I$ a quiver algebra that is Koszul. Suppose that $\eta :\K_n\rightarrow \Lambda$ is a cocycle such that  $ \eta = \begin{pmatrix} 0 & \cdots & 0 & (f^1_w)^{(i)} & 0 & \cdots  & 0\end{pmatrix}$, $0\leq w\leq t_1$. If $\psi_\eta:\K\rightarrow \K[1-n]$ is defined as $\psi_\eta(\varepsilon^m_r) = b_{m,r}(m-n+1,r')\varepsilon^{m-n+1}_{r'}$ and the recursion equation \eqref{big-equation} hold, then $d\psi_\eta - (-1)^{1-n}\psi_\eta d = (\eta\ot 1 - 1 \ot \eta)\Delta_{\K} $.
\end{lemma}

\begin{proof}
We prove this result in the following way; for the free modules $\K_{m-1},\K_m,\K_{m-n},\K_{m-n+1}$, we define the special case of the map $\psi_\eta$ given by Equation \eqref{map-special}. We then use the left and right hand side of \eqref{homo-defi-defi} to derive the recurrence relations. This is equivalent to saying that under these conditions, equation \eqref{homo-defi-defi} holds provided the recurrence relations of \eqref{big-equation} hold.

Let us suppose we have a quiver $Q$ generated by two arrows $\{f^1_r,f^1_s\}$ and each $\K_n$ is free of rank 2. For each $m$ let $\{\varepsilon^{m-1}_{\bar{r}},\varepsilon^{m-1}_{\bar{s}}\}$, $\{\varepsilon^m_r,\varepsilon^m_s\}$, $\{\varepsilon^{m-n+1}_{r'},\varepsilon^{m-n+1}_{s'}\}$, and $\{\varepsilon^{m-n}_{r''},\varepsilon^{m-n}_{s''}\}$ be a basis for $\K_{m-1}$, $\K_m$, $\K_{m-n+1}$, and $\K_{m-n}$ respectively. A possible example of this scenario is given in Example \eqref{example2}.
 The differential given by \eqref{diff-k} on $\varepsilon^m_r$ for this special case for instance, is given by
\begin{align*}
& d(\varepsilon^m_r) = c_{r\bar{r}}(m,r,1)f^1_r\varepsilon^{m-1}_{\bar{r}} + c_{\bar{r}r}(m,r,m-1)\varepsilon^{m-1}_{\bar{r}} f^1_r \\
&+ c_{r\bar{s}}(m,r,1)f^1_r\varepsilon^{m-1}_{\bar{s}} + c_{\bar{r}s}(m,r,m-1)\varepsilon^{m-1}_{\bar{r}} f^1_s + c_{s\bar{r}}(m,r,1)f^1_s\varepsilon^{m-1}_{\bar{r}} \\
&+ c_{\bar{s}r}(m,r,m-1)\varepsilon^{m-1}_{\bar{s}} f^1_r + c_{s\bar{s}}(m,r,1)f^1_s\varepsilon^{m-1}_{\bar{s}} + c_{\bar{s}s}(m,r,m-1)\varepsilon^{m-1}_{\bar{s}} f^1_s
\end{align*}
and a similar formula can be written for $d(\varepsilon^m_s)$.\\
Let us recall that $ \eta = \begin{pmatrix} 0 & \cdots & 0 & (f^1_w)^{(i)} & 0 & \cdots  & 0\end{pmatrix}$ means that $\eta(\varepsilon^n_i)=f^1_w$ with $w=r$ or $w=s$ and $\eta(\varepsilon^n_j)=0$ for all $j\neq i$. From the hypothesis, we define $\psi_\eta:\K_m\rightarrow \K_{m-n+1}$ by $\psi_\eta (\varepsilon^m_r) = b_{m,r}(m-n+1,r')\varepsilon^{m-n+1}_{r'},$ and $\psi_\eta (\varepsilon^m_s) = b_{m,s}(m-n+1,s')\varepsilon^{m-n+1}_{s'},$ and $\psi_\eta:\K_{m-1}\rightarrow \K_{m-n}$ is defined by $\psi_\eta (\varepsilon^{m-1}_{\bar{r}}) = b_{m-1,\bar{r}}(m-n,r'')\varepsilon^{m-n}_{r''},$ and $\psi_\eta (\varepsilon^{m-1}_{\bar{s}}) = b_{m-1,\bar{s}}(m-n,s'')\varepsilon^{m-n}_{s''}$. 

Using Equation\eqref{homo-defi-defi}, the expression $(d\psi_\eta - (-1)^{m-1}\psi_\eta d)(\varepsilon^m_r)$ becomes $d(b_{m,r}(m-n+1,r')\varepsilon^{m-n+1}_{r'}) - (-1)^{m-1}\psi_\eta d(\varepsilon^m_r)$ and is equal to
\begin{align*}
&b_{m,r}(m-n+1,r')\Big( c_{rr''}(m-n+1,r',1)f^1_r\varepsilon^{m-n}_{r''} \\
&+ c_{r''r}(m-n+1,r',m-n)\varepsilon^{m-n}_{r''} f^1_r + c_{rs''}(m-n+1,r',1)f^1_r\varepsilon^{m-n}_{s''} \\
&+ c_{r''s}(m-n+1,r',m-n)\varepsilon^{m-n}_{r''} f^1_s + c_{sr''}(m-n+1,r',1)f^1_s\varepsilon^{m-n}_{r''} \\
& + c_{s''r}(m-n+1,r',m-n)\varepsilon^{m-n}_{s''} f^1_r + c_{ss''}(m-n+1,r',1)f^1_s\varepsilon^{m-n}_{s''} \\
&+ c_{s''s}(m-n+1,r',m-n)\varepsilon^{m-n}_{s''} f^1_s\Big) \\
&-(-1)^{m-1}\Big( b_{m-1,\bar{r}}(m-n,r'') c_{r\bar{r}}(m,r,1)f^1_r\varepsilon^{m-n}_{r''} \\
&+ b_{m-1,\bar{r}}(m-n,r'') c_{\bar{r}r}(m,r,m-1)\varepsilon^{m-n}_{r''} f^1_r \\
&+b_{m-1,\bar{s}}(m-n,s'') c_{r\bar{s}}(m,r,1)f^1_r\varepsilon^{m-n}_{s''} \\
&+ b_{m-1,\bar{r}}(m-n,r'') c_{\bar{r}s}(m,r,m-1)\varepsilon^{m-n}_{r''} f^1_s \\
&+ b_{m-1,\bar{r}}(m-n,r'') c_{s\bar{r}}(m,r,1)f^1_s \varepsilon^{m-n}_{r''} \\
& +b_{m-1,\bar{s}}(m-n,s'') c_{\bar{s}r}(m,r,m-1)\varepsilon^{m-n}_{s''}f^1_r\\
&+  b_{m-1,\bar{s}}(m-n,s'') c_{s\bar{s}}(m,r,1)f^1_s\varepsilon^{m-n}_{s''} \\
&+  b_{m-1,\bar{s}}(m-n,s'') c_{\bar{s}s}(m,r,m-1)\varepsilon^{m-n}_{s''}f^1_s \Big).
\end{align*}
On the other hand,  the diagonal map is given by\\
 $\displaystyle{ \Delta_{\K}(\varepsilon^m_r) =  \sum_{x+y=r}\sum_{u+v=m} c_{x,y}(m,r,u)\varepsilon^u_x\ot_{\Lambda}\varepsilon^{v}_y}.$ We obtain a non-zero in the expansion of $(\eta\ot 1 -1\ot\eta)\Delta_{\K}(\varepsilon^m_r)$ whenever $x=i$ and $y=i$. This means that
\begin{align*}
&(\eta\ot 1 -1\ot\eta) \sum_{x+y=r}\sum_{u+v=m} c_{x,y}(m,r,u)\varepsilon^u_x\ot_{\Lambda}\varepsilon^{v}_y \\
&= (\eta\ot 1)(c_{i,y}(m,r,n)\varepsilon^n_{i}\ot_{\Lambda}\varepsilon^{m-n}_{y}) - (1\ot\eta)(c_{x,i}(m,r,m-n)\varepsilon^{m-n}_{x}\ot_{\Lambda}\varepsilon^{n}_i)\\
&= c_{i,y}(m,r,n)f^1_w\varepsilon^{m-n}_{y} - (-1)^{n(m-n)} c_{xi}(m,r,m-n)\varepsilon^{m-n}_{x}f^{1}_w,
\end{align*}
for $\{x,y\}=\{r'',s''\}$ with $i+y=r, x+i=r$ and some arrow $f^1_w$. After collecting common terms, the expression $(d\psi_\eta - (-1)^{m-1}\psi_\eta d)(\varepsilon^m_r)$ which is the left hand side of Equation \eqref{homo-defi-defi} becomes 
\begin{align*}
&\Big(b_{m,r}(m-n+1,r') c_{rr''}(m-n+1,r',1) \\
&-(-1)^{m-1}b_{m-1,\bar{r}}(m-n,r'') c_{r\bar{r}}(m,r,1) \Big) f^1_r\varepsilon^{m-n}_{r''} \\
&+\Big( b_{m,r}(m-n+1,r') c_{r''r}(m-n+1,r',m-n)\\
& -(-1)^{m-1} b_{m-1,\bar{r}}(m-n,r'') c_{\bar{r}r}(m,r,m-1) \Big) \varepsilon^{m-n}_{r''} f^1_r \\
&+\Big( b_{m,r}(m-n+1,r') c_{rs''}(m-n+1,r',1) \\
&-(-1)^{m-1} b_{m-1,\bar{s}}(m-n,s'') c_{r\bar{s}}(m,r,1) \Big) f^1_r\varepsilon^{m-n}_{s''} \\
&+\Big( b_{m,r}(m-n+1,r') c_{r''s}(m-n+1,r',m-n)\\
& -(-1)^{m-1}  b_{m-1,\bar{r}}(m-n,r'') c_{\bar{r}s}(m,r,m-1) \Big) \varepsilon^{m-n}_{r''} f^1_s \\
&+\Big( b_{m,r}(m-n+1,r') c_{sr''}(m-n+1,r',1)\\
& -(-1)^{m-1} b_{m-1,\bar{r}}(m-n,r'') c_{s\bar{r}}(m,r,1) \Big)f^1_s\varepsilon^{m-n}_{r''} \\
&+\Big( b_{m,r}(m-n+1,r') c_{s''r}(m-n+1,r',m-n)\\
&-(-1)^{m-1} b_{m-1,\bar{s}}(m-n,s'') c_{\bar{s}r}(m,r,m-1) \Big)\varepsilon^{m-n}_{s''} f^1_r \\
&+\Big( b_{m,r}(m-n+1,r') c_{ss''}(m-n+1,r',1)\\
& -(-1)^{m-1} b_{m-1,\bar{s}}(m-n,s'') c_{s\bar{s}}(m,r,1) \Big) f^1_s\varepsilon^{m-n}_{s''} \\
&+\Big( b_{m,r}(m-n+1,r') c_{s''s}(m-n+1,r',m-n) \\
&-(-1)^{m-1}b_{m-1,\bar{s}}(m-n,s'') c_{\bar{s}s}(m,r,m-1) \Big)\varepsilon^{m-n}_{s''} f^1_s.
\end{align*}

The expression $(\eta\ot 1 -1\ot\eta)\Delta_{\K}(\varepsilon^m_r)$ which is the right hand side of Equation \eqref{homo-defi-defi} still remains $\displaystyle{ c_{i,y}(m,r,n)f^1_w\varepsilon^{m-n}_{y} - (-1)^{n(m-n)} c_{xi}(m,r,m-n)\varepsilon^{m-n}_{x}f^{1}_w}$.
We observe the following about indices $w,x,y$. The index $w$ is either $r$ or $s$, the index $y$ is either $r''$ or $s''$ and the index $x$ is either $r''$ or $s''$. We notice that the additional constraint that $i+y=r$ and $i+x=r$ implies that whenever $y=r''$ we must have $x=r''$ and whenever $y=s''$, we must have $x=s''$. We therefore have the following four cases:\\
\textbf{Case I:} Whenever $w=r$, $y=x=r''$, $i+r''=r$, we have the following set of \textit{recurrence} relations on the scalars $b_{m,r}(m-n+1,r'),$
\begin{align*}
& b_{m,r}(m-n+1,r') c_{rr''}(m-n+1,r',1)\\
& -(-1)^{m-1}b_{m-1,\bar{r}}(m-n,r'') c_{r\bar{r}}(m,r,1) =  c_{ir''}(m,r,n) \\
& b_{m,r}(m-n+1,r') c_{r''r}(m-n+1,r',m-n)\\
& -(-1)^{m-1} b_{m-1,\bar{r}}(m-n,r'') c_{\bar{r}r}(m,r,m-1)\\
&\qquad\qquad =   -(-1)^{n(m-n)}c_{r''i}(m,r,m-n)  \\
&b_{m,r}(m-n+1,r') c_{rs''}(m-n+1,r',1)\\
& -(-1)^{m-1} b_{m-1,\bar{s}}(m-n,s'') c_{r\bar{s}}(m,r,1)  = 0
\end{align*}
\begin{align*}
&b_{m,r}(m-n+1,r') c_{r''s}(m-n+1,r',m-n) \\
&-(-1)^{m-1}  b_{m-1,\bar{r}}(m-n,r'') c_{\bar{r}s}(m,r,m-1) = 0 \\
& b_{m,r}(m-n+1,r') c_{sr''}(m-n+1,r',1)\\
& -(-1)^{m-1} b_{m-1,\bar{r}}(m-n,r'') c_{s\bar{r}}(m,r,1) =0 \\
& b_{m,r}(m-n+1,r') c_{s''r}(m-n+1,r',m-n)\\
&-(-1)^{m-1} b_{m-1,\bar{s}}(m-n,s'') c_{\bar{s}r}(m,r,m-1) = 0 \\
& b_{m,r}(m-n+1,r') c_{ss''}(m-n+1,r',1) \\
&-(-1)^{m-1} b_{m-1,\bar{s}}(m-n,s'') c_{s\bar{s}}(m,r,1) = 0\\
& b_{m,r}(m-n+1,r') c_{s''s}(m-n+1,r',m-n) \\
&-(-1)^{m-1}b_{m-1,\bar{s}}(m-n,s'') c_{\bar{s}s}(m,r,m-1) = 0.
\end{align*}
We note that all the equations of Case (I) above can be expressed more succinctly to mean that whenever $w=r$, $i+r''=r$ and for all $s\neq r$
\begin{multline*}
 b_{m,r}(m-n+1,r') c_{rr''}(m-n+1,r',1) \\
 = (-1)^{m-1}b_{m-1,\bar{r}}(m-n,r'') c_{r\bar{r}}(m,r,1) +  c_{ir''}(m,r,n),\\
 b_{m,r}(m-n+1,r') c_{r''r}(m-n+1,r',m-n) \\
= (-1)^{m-1} b_{m-1,\bar{r}}(m-n,r'') c_{\bar{r}r}(m,r,m-1)\\
 -(-1)^{n(m-n)}c_{r''i}(m,r,m-n),
\end{multline*}
and for every indices such that $(p,q)\neq (r,r'')$ and $(p,q)\neq (r'',r)$,
$b_{m,r}(m-n+1,r') c_{pq}(m-n+1,r',*)  = (-1)^{m-1} b_{m-1,\bar{*}}(m-n,*'') c_{pq}(m,r,*).$\\
\textbf{Case II:} Whenever $w=s$, $y=x=r''$, $i+r''=r$, we have the following set of \textit{recurrence} relations on the scalars $b_{m,r}(m-n+1,r'),$
\begin{multline*}
 b_{m,r}(m-n+1,r') c_{sr''}(m-n+1,r',1) \\
 = (-1)^{m-1}b_{m-1,\bar{r}}(m-n,r'') c_{r\bar{r}}(m,r,1) +  c_{ir''}(m,r,n),\\
 b_{m,r}(m-n+1,r') c_{r''s}(m-n+1,r',m-n)\hspace{4cm} \\
= (-1)^{m-1} b_{m-1,\bar{r}}(m-n,r'') c_{\bar{r}r}(m,r,m-1) -(-1)^{n(m-n)}c_{r''i}(m,r,m-n),
\end{multline*}
and for every indices $(p,q)\neq (s,r''), (p,q)\neq (r'',s)$, \\
$b_{m,r}(m-n+1,r') c_{pq}(m-n+1,r',*)  = (-1)^{m-1} b_{m-1,\bar{*}}(m-n,*'') c_{pq}(m,r,*).$\\
\textbf{Case III:} Whenever $w=r$, $y=x=s''$, $i+s''=r$, we have the following set of \textit{recurrence} relations on the scalars $b_{m,r}(m-n+1,r'),$
\begin{multline*}
 b_{m,r}(m-n+1,r') c_{rs''}(m-n+1,r',1)  \\
= (-1)^{m-1}b_{m-1,\bar{s}}(m-n,s'') c_{r\bar{s}}(m,r,1) +  c_{is''}(m,r,n),\\
 b_{m,r}(m-n+1,r') c_{s''r}(m-n+1,r',m-n) \hspace{4cm} \\
= (-1)^{m-1} b_{m-1,\bar{s}}(m-n,s'') c_{\bar{s}r}(m,r,m-1) -(-1)^{n(m-n)}c_{s''i}(m,r,m-n),
\end{multline*}
and for every indices $(p,q)\neq (r,s''), (p,q)\neq (s'',r)$ \\
$b_{m,r}(m-n+1,r') c_{pq}(m-n+1,r',*)  = (-1)^{m-1} b_{m-1,\bar{*}}(m-n,*'') c_{pq}(m,r,*)$.\\

\textbf{Case IV:} Whenever $w=s$, $y=x=s''$, $i+s''=r$, we have the following set of \textit{recurrence} relations on the scalars $b_{m,r}(m-n+1,r'),$
\begin{multline*}
 b_{m,r}(m-n+1,r') c_{ss''}(m-n+1,r',1)  \\
= (-1)^{m-1}b_{m-1,\bar{s}}(m-n,s'') c_{s\bar{s}}(m,r,1) +  c_{is''}(m,r,n),\\
 b_{m,r}(m-n+1,r') c_{s''s}(m-n+1,r',m-n) \hspace{4cm}\\
= (-1)^{m-1} b_{m-1,\bar{s}}(m-n,s'') c_{\bar{s}s}(m,r,m-1)-(-1)^{n(m-n)}c_{s''i}(m,r,m-n),
\end{multline*}
and for every indices $(p,q)\neq (s,s''), (p,q)\neq (s'',s)$\\
$b_{m,r}(m-n+1,r') c_{pq}(m-n+1,r',*)  = (-1)^{m-1} b_{m-1,\bar{*}}(m-n,*'') c_{pq}(m,r,*).$\\

More generally, if $\K_{m}$ is generated by $\{\varepsilon^{m}_r\}_{r=1}^{t_{m}}$, $\K_{m-1}$ generated by $\{\varepsilon^{m-1}_{\bar{r}}\}_{\bar{r}=1}^{t_{m-1}}$, $\K_{m-n}$ by $\{\varepsilon^{m-n}_{r''}\}_{r''=1}^{t_{m-n}}$, and $\K_{m-n+1}$ by $\{\varepsilon^{m-n+1}_{r'}\}_{r'=1}^{t_{m-n+1}}$, the following relations hold for all $r,r',r''$ and $\bar{r}$
\begin{multline*}
(i)\; b_{m,r}(m-n+1,r') c_{rr''}(m-n+1,r',1) \\
 = (-1)^{m-1}b_{m-1,\bar{r}}(m-n,r'') c_{r\bar{r}}(m,r,1) +  c_{ir''}(m,r,n),\\
 b_{m,r}(m-n+1,r') c_{r''r}(m-n+1,r',m-n) \hspace{4cm}\\
= (-1)^{m-1} b_{m-1,\bar{r}}(m-n,r'') c_{\bar{r}r}(m,r,m-1) -(-1)^{n(m-n)}c_{r''i}(m,r,m-n),
\end{multline*}
and for every pair of indices $(p,q)\neq (r,r'')$, $(p,q)\neq (r'',r)$, \\
(ii) $b_{m,r}(m-n+1,r') c_{pq}(m-n+1,r',*)  = (-1)^{m-1} b_{m-1,\bar{*}}(m-n,*'') c_{pq}(m,r,*).$
\end{proof}

\begin{theorem}\label{homotopylifting1-1}
Let $Q$ be a finite quiver and $\Lambda=kQ/I$ a quiver algebra that is Koszul. Suppose that $\eta :\K_n\rightarrow \Lambda$ is a cocycle such that  $ \eta = \begin{pmatrix} 0 & \cdots & 0 & (f^1_w)^{(i)} & 0 & \cdots  & 0\end{pmatrix}$, $0\leq w\leq t_1$ and $f^1_w$ is path of length 1. There are scalars $ \lambda_{m,r}(m-n+1,r')$ such that the map $\psi_\eta :\K\rightarrow \K[1-n]$ associated to $\eta$ and defined in degree $m$ by
$$\psi_\eta (\varepsilon^m_r) = \lambda_{m,r}(m-n+1,r')\varepsilon^{m-n+1}_{r'}$$
is a homotopy lifting map for $\eta$. 
\end{theorem}

\begin{proof}
Let $\lambda_{m,r}(m-n+1,r') = b_{m,r}(m-n+1,r')$ be the scalars satisfying the recurrence relations of \eqref{big-equation}. By Lemma \ref{lemma-homotopylifting-1}, the equation 
$d\psi_\eta - (-1)^{1-n}\psi_\eta d = (\eta\ot 1 - 1 \ot \eta)\Delta_{\K}$ holds, so $\psi_\eta$ is a homotopy lifting map.
\end{proof}

\begin{theorem}\label{homotopylifting1-2}
Let $Q$ be a finite quiver and let $\Lambda=kQ/I$ be a quiver algebra that is Koszul. Suppose that $\eta :\K_n\rightarrow \Lambda$ is a cocycle such that  $ \eta = \begin{pmatrix} 0 & \cdots & 0 & (f^1_wf^1_{w'})^{(i)} & 0 & \cdots  & 0\end{pmatrix}$ for some $0\leq w,w'\leq t_1$ where $f^1_w$ and $f^1_{w'}$ are paths of length 1. Then there exist scalars $\lambda_{m,r}(m-n+1,s)$ and $\lambda_{m,r}(m-n+1,s')$ such that $\psi_\eta :\K_{m} \rightarrow \K_{m-n+1}$ defined by
$$\psi_\eta (\varepsilon^m_r) = \lambda_{m,r}(m-n+1,s)f^1_{w}\varepsilon^{m-n+1}_{s} + \lambda_{m,r}(m-n+1,s')\varepsilon^{m-n+1}_{s'} f^1_{w'}$$
for all $\varepsilon^m_r$ is a homotopy lifting map for $\eta$.
\end{theorem}
\begin{proof}
Similar to Lemma \ref{lemma-homotopylifting-1}, we can write a recurrence relations on the scalars given in Equation \eqref{map-special-2} so that 
$d\psi_\eta - (-1)^{1-n}\psi_\eta d = (\eta\ot 1 - 1 \ot \eta)\Delta_{\K}$ holds true. See \cite[Lemma 5.20]{TNT} and \cite[Theorem 5.23]{TNT} for details about this.
\end{proof}

\subsection{A special case of Theorem \ref{homotopylifting1-1}} We consider a special case in which each $\Lambda^e$-module $\K_n$ is generated by one element. This case arises for example, from a quiver with one arrow (a loop) on a vertex $e_1$. We also give a concrete example in Example \eqref{example1}. Let $I=(x^n)$ be an ideal of the path algebra $kQ$. The quiver algebra of interest here is Morita equivalent the truncated polynomial ring $A = k[x]/(x^n)$. This is the case where $f^n_r=x^n$ where $r=0$ for all $n$ and $\varepsilon^n_r = 1\ot \widetilde{f^n_r}\ot 1$. From the Preliminaries \eqref{prelim}, there are scalars $c_{p,q}(m,r,u)$ for which the diagonal map is given by
\begin{equation}
\Delta_{\K}(\varepsilon^m_r) = \sum_{u+v=m} c_{i,j}(m,r,u)\varepsilon^u_i\ot_{\Lambda}\varepsilon^{v}_j, 
\end{equation}
with $i=j=r=0$. Also with $p=r=0$, the differential takes the form
$$d(\varepsilon^m_r) = c_{p,r}(m,r,1)f^1_p\varepsilon^{m-1}_r + (-1)^m c_{r,p}(m,r,m-1)\varepsilon^{m-1}_rf^1_p.$$
Let $\chi:\K_n\rightarrow A$ be an $n$-cocycle defined by $\chi(\varepsilon^n_r)=f^1_r$. According to Theorem \eqref{homotopylifting1-1}, a homotopy lifting map for $\chi$ can be given by
$$\psi_{\chi_m}(\varepsilon^m_r) = b_{m,r}(m-n+1,r)\varepsilon^{m-n+1}_r,\qquad r=0.$$
We can determine $b_{m,r}(m-n+1,r)$ from the previous scalar $b_{m-1,r}(m-n,r)$. In order words, the conditions (i) and (ii) of Equation \eqref{big-equation} is a \textit{recurrence relation}. We know from Defintion \eqref{homolift} that homotopy lifting maps satisfy
\begin{align*}
&(d\psi_\chi - (-1)^{m-1}\psi_\chi d)(\varepsilon^m_r) = (\chi\ot 1 - 1\ot\chi)\Delta_{\K}(\varepsilon^m_r), \qquad\text{so then}\\
& d( b_{m,r}(m-n+1,r)\varepsilon^{m-n+1}_r) - (-1)^{m-1}\psi_\chi \\
&\qquad \Big(c_{p,r}(m,r,1)f^1_p\varepsilon^{m-1}_r + (-1)^m c_{r,p}(m,r,m-1)\varepsilon^{m-1}_rf^1_p\Big) \\
&= (\chi\ot 1 - 1\ot\chi) \sum_{u+v=m} c_{i,j}(m,r,u)\varepsilon^u_i\ot_{\Lambda}\varepsilon^{v}_j.
\end{align*}
The modules $\K$ are generated by one elements so we get
\begin{multline*}
 b_{m,r}(m-n+1,r) c_{p,r}(m-n+1,r,1)f^1_p\varepsilon^{m-n}_r + (-1)^{m-n+1}b_{m,r}(m-n+1,r)\\
 c_{r,p}(m-n+1,r,m-n)\varepsilon^{m-n}_rf^1_p  - (-1)^{m-1} b_{m-1,r}(m-n,r) c_{p,r}(m,r,1)f^1_p\varepsilon^{m-n}_r \\
+ b_{m-1,r}(m-n,r) c_{r,p}(m,r,m-1)\varepsilon^{m-n}_rf^1_p \\
= c_{r,j}(m,r,n)f^1_r\varepsilon^{m-n}_j + (-1)^{n(m-n)} c_{i,r}(m,r,m-n)\varepsilon^{m-n}_if^{1}_r
\end{multline*}
We would obtain the following expressions for the above equality to hold,
\begin{align*}
&b_{m,r}(m-n+1,r) c_{p,r}(m-n+1,r,1)- (-1)^{m-1} b_{m-1,r}(m-n,r) c_{p,r}(m,r,1) \\
&= c_{p,r}(m,r,n)\quad\text{and} \\
&(-1)^{m-n+1}b_{m,r}(m-n+1,r) c_{r,p}(m-n+1,r,m-n) \\
&+ b_{m-1,r}(m-n,r) c_{r,p}(m,r,m-1) = (-1)^{n(m-n)} c_{r,p}(m,r,m-n).
\end{align*}
The scalars $c_{p,r}(m-n+1,r,*)$ come from the differentials on the resolution $\K$, so they are not equal to $0$ for all $r$. In case $c_{p,r}(m-n+1,r,*)\neq 0$ for all $r$, the first equality in the last expression yields 
\begin{equation}\label{recur-b}
b_{m,r}(m-n+1,r)  = \frac{(-1)^{m-1} b_{m-1,r}(m-n,r) c_{p,r}(m,r,1) +  c_{p,r}(m,r,n)}{c_{p,r}(m-n+1,r,1)}
\end{equation}
while the second one yields
\begin{multline}\label{recur-bb}
b_{m,r}(m-n+1,r)  \\
= \frac{ b_{m-1,r}(m-n,r) c_{r,p}(m,r,m-1)  +  (-1)^{n(m-n)+1} c_{r,p}(m,r,m-n)}{(-1)^{m-n} c_{r,p}(m-n+1,r,m-n)}.
\end{multline}

We now present the Gerstenhaber bracket structure on Hochschild cohomology using these scalars.

\begin{theorem}\label{Gerstenbrack1}
Let $\Lambda=kQ/I$ be a quiver algebra that is Koszul. Denote by $\{f^m_r\}_{r=0}^{t_m}$ elements of $kQ$ defining a minimal projective resolution of $\Lambda_0$ as a right $\Lambda$-module. Let $\K$ be the projective bimodule resolution of $\Lambda$ with $\K_m$ having basis $\{\varepsilon^m_r\}_{r=0}^{t_m}.$  Assume that $\eta:\K_n\rightarrow\Lambda$ and $\theta:\K_m\rightarrow\Lambda$ represent elements in $\HH^*(\Lambda)$ and are given by $\eta(\varepsilon^n_i) = \lambda_i$ for $i=0,1,\ldots,t_n$ and $\theta(\varepsilon^m_j) = \beta_j$ for $j=0,1,\ldots,t_m,$ with each $\lambda_i$ and $\beta_j$ paths of length of 1. Then the class of the bracket $[\eta, \theta] : \K_{n+m-1}\rightarrow\Lambda$ can be expressed on the $r$-th basis element $\varepsilon^{m+n-1}_r$ as
\begin{align*}
&[\eta, \theta](\varepsilon^{m+n-1}_r) = \sum_{i=0}^{t_n} \sum_{j=0}^{t_m}  b_{m-n+1,r}(n,i)\lambda_i   - (-1)^{(m-1)(n-1)}(b_{m-n+1,r}(m,j)\beta_j
\end{align*}
for some scalars $b_{m-n+1,r}(n,i)$ and $b_{m-n+1,r}(m,j)$ associated with homotopy lifting maps $\psi_{\theta^{(j)}}$ and $\psi_{\eta^{(i)}}$ respectively.
\end{theorem}
\begin{proof}
This is same as \cite[Theorem 3.15]{TNO2} and proved therein.
\end{proof}

\section{Some computations and examples}\label{example-section}
In this section, we give examples in which the scalars $b_{m,r}(m-n+1,*)$ are obtained from  $b_{m-1,r}(m-n,**)$ using the recurrence relations of Theorem \eqref{homotopylifting1-1}, Equations $\eqref{recur-b}$ and $\eqref{recur-bb}$. In most of the examples, we described the scalars $c_{p,q}(m,r,n)$ which are also used in the recurrence relations. 

\begin{example}\label{example1} Let's consider the following quiver
$$Q:= \begin{tikzcd} 1 \arrow[out=100,in=200,loop,swap,"x"]\end{tikzcd},$$
 and take $A=k[x]/(x^n)$ to be the truncated polynomial ring. The idempotent $f^0_0=1$, generates $\K_0$, $f^1_0=x$ generates $\K_1$ and $f^n_0=x^n$ generates $\K_n$. Notice that we can write $f^n_0 =  f^1_0f^{n-1}_0 = f^{n-1}_0f^{1}_0$, so then $c_{0,0}(n,0,1)= c_{0,0}(n,0,n-1)=1.$ If we \textbf{assume that $n=2$}, that is, we mod out by the ideal $I=(x^2)$, some calculations show that using $\varepsilon^n_0=1\ot f^n_0\ot 1$, 
\begin{align*}
d(\varepsilon^1_0) &= d(1\ot x\ot 1) = x(1\ot 1) - (1\ot 1)x =  x\varepsilon^0_0 -\varepsilon^0_0 x\\
d(\varepsilon^2_0) &= d(1\ot x\ot x\ot 1) = x(1\ot x\ot 1) - 1\ot x^2\ot 1 + (1\ot x\ot 1)x = x\varepsilon^1_0 +\varepsilon^1_0 x\\
d(\varepsilon^3_0) &= d(1\ot x\ot x\ot x\ot 1) = x(1\ot x\ot x\ot 1) - (1\ot x\ot x\ot 1)x \\
&= x\varepsilon^1_0 -\varepsilon^1_0 x \quad\text{and more generally}\\
d(\varepsilon^n_0) &= x\varepsilon^{n-1}_0 -(-1)^{n-1}\varepsilon^{n-1}_0 x = f^1_0\varepsilon^{n-1}_0 -(-1)^{n-1}\varepsilon^{n-1}_0 f^1_0,
\end{align*}
Let $\eta:\K_1\rightarrow A$ be defined by $\eta(\varepsilon^1_0)=x$. Also let $\chi:\K_2\rightarrow A$ be defined by $\chi(\varepsilon^2_0)=x$. A short calculation shows that $\eta$ and $\chi$ are cocycles. A diagonal map for this particular resolution is given by  $\Delta_{\K}(\varepsilon^m_0) = \sum_{i+j=m} \varepsilon^i_0\ot\varepsilon^{j}_0.$ It can be verified by direct evaluation of Equation \eqref{homo-defi} that the map 
$$\psi_\eta:\K_m\rightarrow\K_m \qquad \text{defined by}\qquad \psi_\eta(\varepsilon^m_0) = m\varepsilon^m_0$$
is a homotopy lifting map for $\eta$ that is
\begin{align}\label{homo-verify}
&(d\psi_\eta-(-1)^0\psi_\eta d)(\varepsilon^m_0) = d(m\varepsilon^m_0)-\psi_\eta(x\varepsilon^{m-1}_0-(-1)^{m-1}\varepsilon^{m-1}_0x)\\ \notag
&= mx\varepsilon^{m-1}_0 - (-1)^{m-1}m\varepsilon^{m-1}_0x - (m-1)x\varepsilon^{m-1}_0+(-1)^{m-1}(m-1)\varepsilon^{m-1}_0x\\ \notag
&= x\varepsilon^{m-1}_0-(-1)^{m-1}\varepsilon^{m-1}_0x \qquad\text{ is equal to}\\ \notag
&(\eta\ot 1 -1\ot\eta)\Delta_{\K}(\varepsilon^m_0) = (\eta\ot 1 -1\ot\eta)\sum_{i+j=m} \varepsilon^i_0\ot\varepsilon^{j}_0\\ \notag
&=\eta\ot 1(\varepsilon^1_0\ot\varepsilon^{m-1}_0) -(-1)^{m-1} 1\ot\eta(\varepsilon^{m-1}_0\ot\varepsilon^{1}_0)\\ \notag
&=x\varepsilon^{m-1}_0-(-1)^{m-1}\varepsilon^{m-1}_0x,
\end{align}
where the Koszul sign convention has been employed in the expansion of $(1\ot\eta)(\varepsilon^{m-1}_0\ot\varepsilon^{1}_0) = (-1)^{degree(\eta)\cdot (m-1)}\varepsilon^{m-1}_0\eta(\varepsilon^{1}_0)$. We note that by the general definition given in Theorem \ref{homotopylifting1-1}, the map $\psi_\eta:\K_{m-1}\rightarrow{}\K_{m-1}$ defined by $\psi_\eta(\varepsilon^{m-1}_0)=(m-1)\varepsilon^{m-1}_0$ implies that $b_{m-1,0}(m-1,0)=m-1$. The map $\eta$ is a $1$-cocycle so $n=1$, $r=p=0.$


We can use the expression of \eqref{recur-bb} to verify that
\begin{align*}
&b_{m,r}(m-n+1,r)   \\
&= \frac{ b_{m-1,r}(m-n,r) c_{r,p}(m,r,m-1)  +  (-1)^{n(m-n)+1} c_{r,p}(m,r,m-n)}{(-1)^{m-n} c_{r,p}(m-n+1,r,m-n)}\\
b_{m,0}(m,0)  &= \frac{ b_{m-1,0}(m-1,0) c_{0,0}(m,0,m-1)  +  (-1)^{m} c_{0,0}(m,0,m-1)}{(-1)^{m-1} c_{0,0}(m,0,m-1)} \\
&= \frac{m-1+(-1)^{m}}{1} = m, \qquad \text{when $m$ is even},
\end{align*}
and the expression of \eqref{recur-b} to verify that
\begin{align*}
b_{m,0}(m,0)  &= \frac{(-1)^{m-1} b_{m-1,0}(m-1,0) c_{0,0}(m,0,1)  +   c_{0,0}(m,0,1)}{c_{0,0}(m,0,1)} \\
&= \frac{m-1+ 1}{1} = m, \qquad \text{when $m$ is odd}.
\end{align*}

Similarly, it is a straightforward calculation (same calculations as \eqref{homo-verify}) to verify that the map 
$\psi_\chi:\K_m\rightarrow\K_{m-1}$ defined by
$$\psi_\chi(\varepsilon^m_0) = b_{m,0}(m-1,0)\varepsilon^{m-1}_0 = \begin{cases} \varepsilon^{m-1}_0, & \rm{ when\; m\;is\;even}\\ 0, &\rm{ when\; m\;is\;odd}\end{cases}$$
is a homotopy lifting map for $\chi$. In this case $b_{m,0}(m-1,0)=1$ when $m$ is even and $0$ when $m$ is odd. But we can also use the expression of \eqref{recur-b} to verify that when $m$ is even,  
$$b_{m+1,0}(m,0) =\frac{(-1)^{m-1} b_{m,0}(m-1,0) c_{0,0}(m,0,1) +  c_{0,0}(m,0,2)}{c_{0,0}(m-n+1,0,1)} = -1+1 = 0,$$ 
and when $m$ is odd,  
$$b_{m+1,0}(m,0) =\frac{(-1)^{m-1} b_{m,0}(m-1,0) c_{0,0}(m,0,1) +  c_{0,0}(m,0,2)}{c_{0,0}(m-n+1,0,1)} =  0+1 = 1.$$ 
\end{example}

\begin{example}\label{example2a}
The following example of a homotopy lifting map was first given in \cite[Example 4.7.2]{NVW}. We will now verify that the recurrence relations also hold. Let $k$ be a field and $A=k[x]/(x^3)$. Consider the following projective bimodule resolution of $A$:
\begin{equation*}
\mathbb{P}_{\bullet}:  \qquad\cdots \rightarrow A\ot A\xrightarrow{\;\cdot u\; }A\ot A\xrightarrow{\cdot u\;} \cdots \xrightarrow{\;\cdot v\;}A\ot A\xrightarrow{\;\cdot u\;}A\ot A\;(\;\xrightarrow{\mu} A)
\end{equation*}
where $u=x\ot 1 -1\ot x$ and $v = x^2\ot 1 + x\ot x + 1\ot x^2$. We consider the following elements $e_m:=1\ot 1$, $r=0$ for all $m$ in the $m$-th module $P_m :=A\ot A$. A diagonal map $\Delta_{\mathbb{P}}: \mathbb{P}\xrightarrow{} \mathbb{P}\ot_A\mathbb{P}$ for this resolution is given by
\begin{equation*}
\Delta_{\mathbb{P}}(e_m) = \sum_{j+l=m} (-1)^{j}e_j\ot e_l.
\end{equation*}
By comparing $\Delta_{\mathbb{P}}(e_m)$ with Equation \eqref{diag-k}, the scalars $c_{rr}(m,r,j) = (-1)^j$ for all $m,r$. Consider the Hochschild 1-cocycle $\alpha:\mathbb{P}_1\xrightarrow{} A$ defined by $\alpha(e_1)=x$ and $\alpha(e_m)=0$ for all $m\neq 1$. With a slight change in notation, it was shown in \cite[Example 4.7.2]{NVW} that the following $\psi_{\alpha}:\mathbb{P}_{2m}\xrightarrow{}\mathbb{P}_{2m}$ defined by $\psi_{\alpha}(e_{2m}) = -3m\cdot e_{2m}$ is a homotopy lifting map for $\alpha$. We note that the map $\psi_{\alpha}$ was regarded as an $A_{\infty}$-coderivation in \cite{NVW}. It can be also verified that $\psi_{\alpha}$ is a homotopy lifting map. We can use the recurrence relations of Equation \eqref{recur-bb} to obtain $b_{2m+1,r}(2m+1,r)$ from $b_{2m,r}(2m,r)=-3m$. That is
\begin{align*}
b_{2m+1,r}(2m+1,r) &=\frac{ b_{2m,r}(2m,r) c_{r,r}(2m,r,2m) + (-1)^{2m+1} c_{r,r}(2m+1,r,2m)}{(-1)^{2m}c_{r,r}(2m+1,r,2m)}\\ &= \frac{-3m(-1)^{2m}+(-1)^{2m+1}(-1)^{2m}}{(-1)^{2m}(-1)^{2m-1+1}} = \frac{-3m-1}{1},
\end{align*}
so it follows that  $\psi_{\alpha}:\mathbb{P}_{2m+1}\xrightarrow{}\mathbb{P}_{2m+1}$ is defined by $\psi_{\alpha}(e_{2m+1}) = (-3m-1) e_{2m+1}.$
\end{example}

\begin{example}\label{example2}
Let $k$ be a field of characteristics different from 2. Consider the quiver algebra $A=kQ/I$ (also examined in \cite[ Example 5]{MSKA}) defined using the following finite quiver:
$$
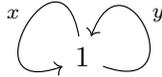
\begin{tikzcd} 1 \arrow[out=100,in=200,loop,swap,"x"]
  \arrow[out=340,in=70,loop,swap,"y"]
\end{tikzcd}$$
with one vertex and two arrows  $x,y.$ We denote by $e_1$ the idempotent associated with the only vertex. Let $I$, an ideal of the path algebra $kQ$ be defined by $\displaystyle{I = \langle x^2,xy+yx\rangle.}$ Since $\{x^2, xy+yx\}$ is a quadratic Grobner basis for the ideal generated by relations under the length lexicographich order with $x>y>1$, the algebra is Koszul.

In order to define a comultiplicative structure, we take $t_0=0, t_n=1$ for all $n$, $f_0^0=e_1, f_1^0=0, f_0^1=x, f_1^1=y, f^2_0=x^2, f^2_1=xy+yx, f_0^3=x^3, f_1^3 = x^2y+xyx+yx^2,$ and in general $ f^n_0 = x^n, f^n_1 = \sum_{i+j=n-1}x^iyx^j$. We also see that $f^n_0 = f^r_0 f^{n-r}_0$ and $f^n_1 = f^r_0f^{n-r}_1 + f^r_1f^{n-r}_0$ so $c_{00}(n,0,r) = c_{01}(n,1,r) = c_{10}(n,1,r) = 1$ and all other $c_{pq}(n,i,r)=0$. With the above stated, we can construct the resolution $\K$ for the algebra $A$. A calculation shows that 
\begin{align*}
&d_1(\varepsilon^1_0) = x\varepsilon^0_0 - \varepsilon^0_0 x, && d_1(\varepsilon^1_1) = y\varepsilon^0_0 - \varepsilon^0_0 y\\
&d_2(\varepsilon^2_0) = x\varepsilon^1_0 + \varepsilon^1_0 x, && d_2(\varepsilon^2_1)= y\varepsilon^1_0 + \varepsilon^1_0 y + x\varepsilon^1_1 + \varepsilon^1_1 x.
\end{align*}
Consider the following map $\theta:\K_1\xrightarrow{}A$ defined by $\theta = (0\;\; y)$. With the following calculations
$\theta d_2(\varepsilon^2_0) = \theta (x\varepsilon^1_0 + \varepsilon^1_0 x) = 0$ and $\theta d_2(\varepsilon^2_1) = \theta (y\varepsilon^1_0 + \varepsilon^1_0 y + x\varepsilon^1_1 + \varepsilon^1_1 x) = 0 + xy + yx = 0$, $\theta$ is a cocycle. The comultiplicative map $\Delta:\K\xrightarrow{}\K\ot_A\K$ on $\varepsilon^1_0$, $\varepsilon^1_1, \varepsilon^2_0,\varepsilon^2_1$ is given by 
\begin{align*}
\Delta(\varepsilon^1_0) &= c_{00}(1,0,0)\varepsilon^0_0\ot\varepsilon^1_0 + c_{00}(1,0,1)\varepsilon^1_0\ot\varepsilon^0_0=  \varepsilon^0_0\ot\varepsilon^1_0 + \varepsilon^1_0\ot\varepsilon^0_0,\\
\Delta(\varepsilon^1_1) &= \varepsilon^0_0\ot\varepsilon^1_1 + \varepsilon^1_1\ot\varepsilon^0_0,\\
\Delta(\varepsilon^2_0) &= \varepsilon^0_0\ot\varepsilon^2_0 + \varepsilon^1_0\ot\varepsilon^1_0 +\varepsilon^2_0\ot\varepsilon^0_0,\\
\Delta(\varepsilon^2_1) &=  \varepsilon^0_0\ot\varepsilon^2_1 + \varepsilon^1_0\ot\varepsilon^1_1 +\varepsilon^1_1\ot\varepsilon^1_0 + \varepsilon^2_1\ot\varepsilon^0_0.
\end{align*}
From Theorems \eqref{homotopylifting1-1}, it can be verified by direct calculations using Equation \eqref{homo-verify} that the first, second and third degrees of the homotopy lifting maps $\psi_\theta$ associated $\theta$ is the following:
\begin{align*}
\psi_{\theta_0}(\varepsilon^0_i) = 0,\qquad & \psi_{\theta_1}(\varepsilon^1_0) = 0, \qquad \psi_{\theta_1}(\varepsilon^1_1) = \varepsilon^1_1 && \psi_{\theta_2}(\varepsilon^2_0) = 0.
\end{align*}
The scalars $b_{1,1}(1,1)= 1$  and for other $(m,r)\neq (1,1), (2,1)$, $b_{m,r}(m,r)=0$. Since $\theta$ is a $1$-cocycle, $n=1$. Also, $\theta = (0\; y)$ and when compared with $\eta = \begin{pmatrix} 0 & \cdots & 0 & (f^1_w)^{(i)} & 0 & \cdots  & 0\end{pmatrix}$ as given in Theorem \eqref{homotopylifting1-1}, $f^1_w=y$ and $i=1$. To obtain $ b_{2,1}(2,1)$ from $ b_{1,r}(1,s)$ for some $s$, we take $m=2, r=1$, so that $m-n=1$. Since $t_{m-n}=t_1$, we would have $r''=0$ or $r''=1$. From the statement of the theorem, we must have $i+r''=r$, so $r''=0$ and $c_{r,r''}(m-n+1,r',1) = c_{10}(2,r',1)=1$. It then follows from the first recurrence relations in Theorem \eqref{homotopylifting1-1} that 
\begin{align*}
b_{m,r}(m-n+1,r')  &= \frac{(-1)^{m-1} b_{m-1,\bar{r}}(m-n,r'') c_{r,\bar{r}}(m,r,1)  +  c_{i,r''}(m,r,n)}{c_{r,r''}(m-n+1,r',1)}.\\
b_{2,1}(2,1)  &=  \frac{- b_{1,\bar{r}}(1,0) c_{1,\bar{r}}(2,1,1)  +  c_{1,0}(2,1,1)}{ c_{1,0}(2,1,1)} = \frac{0+1}{1}=1,
\end{align*}
so  $\psi_{\theta_2}(\varepsilon^2_1) = b_{2,1}(2,1)\varepsilon^2_1 =\varepsilon^2_1.$ 
\end{example}

\section{Finding Maurer-Cartan elements}\label{finding-mc}
In this section, we find the Maurer-Cartan elements of a quiver algebra. We first recall the definition of a Maurer-Cartan element.
\begin{definition}\label{MC-defi}
A Hochschild 2-cocycle $\eta$ is said to satisfy the Maurer-Cartan equation if
\begin{equation}\label{MC-equation}
d(\eta) + \frac{1}{2}[\eta,\eta] = 0.
\end{equation}
Applying the definition of the bracket using homotopy lifting, we obtain the following version of the Maurer-Cartan equation for the resolution $\K$.
$ d^{*}_{3}(\eta) + \frac{1}{2}(\eta\psi_\eta + \eta\psi_\eta) = d^{*}_{3}(\eta) + \eta\psi_\eta = 0.$
\end{definition}

We begin with the following finite quiver:
$$Q:=  \begin{tikzcd} 1 \arrow[out=190,in=270,loop,swap,"b"]
  \arrow[out=90,in=170,loop,swap,"a"]
  \arrow[r,"c"] & 2
\end{tikzcd}$$
with two vertices and three arrows  $a,b,c.$ We denote by $e_1$ and $e_2$ the idempotents associated with vertices 1 and 2. Let $kQ$ be the path algebra associated with $Q$ and take for each $q\in k$, $I_q\subseteq kQ$ to be an admissible ideal of $kQ$ generated by $I_q = \langle a^2,b^2,ab-qba, ac\rangle$ so that
$$\{A_q = kQ/I_q\}_{q\in k}$$
This family of quiver algebras have been well studied in \cite{TNO, TNO2} and \cite{TNT}. We simply recall the main tools needed to find Maurer-Cartan elements. To define a set of generators for the resolution $\K$ we start by letting $kQ_0$ to be the ideal of $kQ$ generated by the vertices of $Q$ with basis $f_0^0=e_1,f_1^0=e_2.$ Next, set $kQ_1$ to be the ideal generated by paths with basis $f_0^1=a,f_1^1=b$ and $f_2^1=c.$ Set $f^2_j,$ $j=0,1,2,3$ to be the set of paths of length 2 that generates the ideal $I$, that is
$f_0^2=a^2,f_1^2=ab-qba,f_2^2=b^2, f^2_3= ac,$ and define a comultiplicative equation on the paths of length $n>2$ in the following way.
$$\begin{cases} 
f^n_0 = a^{n}, \\
f^n_s = f^{n-1}_{s-1} b + (-q)^s f^{n-1}_s a ,& (0<s<n), \\
f^n_n = b^{n}, \\
f^n_{n+1} = a^{(n-1)} c,
\end{cases}$$
The resolution $\K \rightarrow A_q$ has basis elements $\{\varepsilon^n_i\}_{i=0}^{t_n}$ such that for each $i$, we have $\varepsilon^n_i = (0,\ldots,0,o(f^n_i)\otimes_k t(f^n_i),0,\ldots,0)$. The differentials on $\K_n$ are given explicitly for this family  by
\begin{align*}
d_1(\varepsilon^1_2) &= c\varepsilon^0_1 - \varepsilon^0_0 c\\
d_n(\varepsilon^n_r) &= (1-\partial_{n,r})[a\varepsilon^{n-1}_r)+(-1)^{n-r}q^r\varepsilon^{n-1}_r a] \\
&+(1-\partial_{r,0})[(-q)^{n-r}b\varepsilon^{n-1}_{r-1}+(-1)^{n}\varepsilon^{n-1}_{r-1} b], \;\;\text{for}\;\;r\leq n\\
d_n(\varepsilon^n_{n+1}) &= a\varepsilon^{n-1}_n + (-1)^n\varepsilon^{n-1}_0 c, \;\;\text{when}\;\;n\geq 2,
\end{align*}
where $\partial_{r,s} = 1$ when $r=s$ and $0$ when $r\neq s$.

Calculations from~\cite{TNO} show that for this family, the comultiplicative map can be expressed in the following way
\begin{equation*}
\Delta_{\K}(\varepsilon^{n}_{s}) =
\begin{cases}
\displaystyle{\sum_{r=0}^{n} \varepsilon^r_0\ot \varepsilon^{n-r}_0,} &  s=0 \\
\displaystyle{ \sum_{w=0}^{n}\sum_{j=max\{0,s+w-n\}}^{min\{w,s\}} (-q)^{j(n-s+j-w)} \varepsilon^w_j\ot \varepsilon^{n-w}_{s-j}, } & 0< s < n\\
\displaystyle{\sum_{t=0}^{n} \varepsilon^t_t\ot \varepsilon^{n-t}_{n-t}, } & s=n \\
\displaystyle{ \varepsilon^0_0\ot \varepsilon^{n}_{n+1} + \Big[ \sum_{t=0}^{n} \varepsilon^t_0\ot \varepsilon^{n-t}_{n-t+1}\Big] +  \varepsilon^n_{n+1}\ot \varepsilon^{0}_{0},} &s=n+1.
\end{cases} 
\end{equation*}

\begin{example}\label{mc-elements}
Let $A_1 = kQ/I_1$ be a member of the family where $I = I_1= \langle a^2,b^2,ab-ba, ac\rangle$. We now find Hochschild 2-cocycles that satisfy the Maurer-Cartan equation of \ref{MC-equation}. Suppose that the $A_1^e$-module homomorphism $\eta:\K_2\rightarrow A_1$ defined by \\
$ \eta
 = \begin{pmatrix} \lambda_{0} & \lambda_{1}& \lambda_{2} &  \lambda_{3}\end{pmatrix}$  is a cocycle, that is $d^*\eta = 0,$ with $\lambda_i\in\Lambda_q$ for all $i$. Since $d^*\eta : \K_3\rightarrow A_1$, we obtain using $d^*\eta(\varepsilon^3_i) = \eta d(\varepsilon^3_i)$, 
$$
\eta \Big(\begin{cases} a\varepsilon^{2}_{0} - \varepsilon^{2}_{0}a & \\
 a\varepsilon^{2}_{1}+\varepsilon^{2}_{1}a +  b\varepsilon^{2}_{0} - \varepsilon^{2}_{0}b & \\ 
 a\varepsilon^{2}_{2}-\varepsilon^{2}_{2}a -  b\varepsilon^{2}_{1} - \varepsilon^{2}_{1}b & \\ 
 b\varepsilon^{2}_{2}- \varepsilon^{2}_{2}b   & \\   
 a\varepsilon^{2}_{3}- \varepsilon^{2}_{0}c  & 
 \end{cases}\Big) = \begin{cases}a\lambda_{0} - \lambda_{0}a,  & {\rm if}\;i=0\\
 a\lambda_{1}+q\lambda_{1}a + q^2 b\lambda_{0} - \lambda_{0}b & {\rm if}\;i=1\\ 
a\lambda_{2}-q^2\lambda_{2}a - q b\lambda_{1} - \lambda_{1}b & {\rm if}\;i=2\\ 
 b\lambda_{2}- \lambda_{2}b   & {\rm if}\;i=3\\   
 a\varepsilon^{2}_{3}- \varepsilon^{2}_{0}c  & {\rm if}\;i=4
 \end{cases}
$$
which will be equated to $\begin{pmatrix} 0 &0&0&0&0\end{pmatrix}$ and solved. We solve this system of equations with the following in mind. There is an isomorphism of $A_1^e$-modules $\HHom_{A_1^e}(A_1 o(f^n_i)\otimes_k t(f^n_i)A_1, A_1) \simeq o(f^n_i)A_1 \; t(f^n_i)$ ensuring that
\begin{align*}
o(f^2_i)\lambda_i t(f^2_i) &= o(f^2_i)\eta(\varepsilon^2_i) t(f^2_i) = o(f^2_i)\eta(o(f^2_i)\otimes_k t(f^2_i)) t(f^2_i)\\
 &= \phi(o(f^2_i)^2\otimes_k t(f^2_i)^2) = \phi(o(f^2_i)\otimes_k t(f^2_i)) = \lambda_i.
 \end{align*}
This means that for $i=0,1,2$ each $\lambda_i$ should satisfy $e_1\lambda_ie_1 = \lambda_i$  since the origin and terminal vertex of $f^2_0, f^2_1, f^2_2$ is $e_1$ and  $e_1\lambda_3e_2 = \lambda_3$. We obtain 9 solutions presented in Table \ref{table 1}.
\begin{table}[t]
\begin{tabular}{ |c|c|c|c|c|c|c|c|c|c| }  \hline
solutions             & 1 & 2  & 3 & 4 & 5 & 6 & 7 & 8 & 9 \\ \hline
$\lambda_0$ & a & ab & 0 & 0 & 0 & 0 & 0 & 0 & 0 \\ 
$\lambda_1$ & 0 & 0   & 0 & 0 & 0 & 0 & ab & 0 & 0 \\
$\lambda_2$ & 0 & 0   & a & b & ab& $e_1$ & 0 & 0 & 0 \\ 
$\lambda_3$ & 0 & 0   & 0 & 0 & 0 & 0 & 0 & c & bc \\  \hline
\end{tabular}
\caption{Possible values of $\eta(\varepsilon^2_r) = \lambda_r$ for different $r$.}
\label{table 1}
\end{table}

Now suppose that there is some $\phi:\K_1\rightarrow A_1$ such that $\phi d_2(\varepsilon^2_i) = \eta(\varepsilon^2_i)$, i=0,1,2,3. If $\phi =  \begin{pmatrix} 0 & \frac{1}{2}a & 0 \end{pmatrix}$, we get $\eta =  \begin{pmatrix} 0 & 0 & ab &  0 \end{pmatrix}$, so  $\eta =  \begin{pmatrix} 0 & 0 & ab &  0 \end{pmatrix} \in \Img(d_2^{*})$. If $\phi$ is equal to  $\begin{pmatrix} 0 & \frac{1}{2}e_1 & 0 \end{pmatrix}$,  $\begin{pmatrix} e_1 & 0 & 0 \end{pmatrix}$ or $\begin{pmatrix} b & 0 & 0 \end{pmatrix}$, we obtain the following for $\eta$; $\begin{pmatrix} 0 & 0 & b &  0 \end{pmatrix}$, $\begin{pmatrix} 0 & 0 & 0 &  c \end{pmatrix}$ and $\begin{pmatrix} 0 & 0 & 0 &  bc \end{pmatrix}$ respectively. Therefore $\HH^2(A_1) = \frac{\Ker d_3^*}{\Img d_2^*}$ is generated as a $k$-vector space by $\langle \widetilde{\eta}, \bar{\eta}, \chi,\bar{\chi},\sigma\rangle $ where
$\widetilde{\eta}= \begin{pmatrix} a & 0 & 0 &  0 \end{pmatrix}$, $\bar{\eta}=\begin{pmatrix} ab & 0 & 0 &  0 \end{pmatrix}$, $\bar{\chi}=\begin{pmatrix} 0 & ab & 0 &  0 \end{pmatrix}$, $\chi=\begin{pmatrix} 0 & 0 & a &  0 \end{pmatrix}$ and $\sigma=\begin{pmatrix} 0 & 0 & e_1 &  0 \end{pmatrix}$.

Given in Table \ref{table 2} are the first, second and third degree homotopy lifting maps associated to each the above elements of $\HH^2(A_1)$. It can be easily verified using the homotopy lifting equation in Definition \eqref{homo-defi} that these indeed are homotopy lifting maps.

\begin{table}[t]
\begin{tabular}{|c c|}\hline
& For $\eta = \begin{pmatrix} a & 0 & 0 &  0 \end{pmatrix},$ we get \\ \hline
& $\psi_{\eta_1} (\varepsilon^1_i) = 0,\;i=0,1,2,$ \\
& $\psi_{\eta_2} (\varepsilon^2_i) = \begin{cases} \varepsilon^1_0, & {\rm if }\;i=0 \\ 0, & {\rm if }\;i=1,2,3 \end{cases}, \quad
\psi_{\eta_3}  (\varepsilon^3_i) = \begin{cases} 0, & {\rm if }\;i=0  \\ \varepsilon^2_{1},& {\rm if }\;i=1 \\ 0,& {\rm if }\;i=2 \\ 0,& {\rm if }\;i=3  \\  \varepsilon^2_{3},& {\rm if }\;i=4 \end{cases} $\\ \hline
& For $\chi = \begin{pmatrix} 0 & 0 & a &  0 \end{pmatrix},$ we get \\ \hline
& $\psi_{\chi_1} (\varepsilon^1_i) = 0,\;i=0,1,2,\;$\\
& $\psi_{\chi_2} (\varepsilon^2_i) = \begin{cases} 0, & {\rm if }\;i=0,1,3 \\  \varepsilon^1_0, & {\rm if }\;i=2 \end{cases},, \quad
\psi_{\chi_3}  (\varepsilon^3_i) = \begin{cases} 0, & {\rm if }\;i=0  \\ 0,& {\rm if }\;i=1 \\ 0,& {\rm if }\;i=2 \\  \varepsilon^2_{1},& {\rm if }\;i=3  \\  0,& {\rm if }\;i=4 \end{cases}$ \\ \hline
& For $\bar{\eta} = \begin{pmatrix} ab & 0 & 0 &  0 \end{pmatrix},$ we get\\ \hline
& $\psi_{\bar{\eta}_1} (\varepsilon^1_i) = 0,\;i=0,1,2,\;$\\
& $\psi_{\bar{\eta}_2} (\varepsilon^2_i) = \begin{cases} a\varepsilon^1_1+\varepsilon^1_0 b & {\rm if }\;i=0 \\ 0, & {\rm if }\;i=1 \\ 0 & {\rm if }\;i=2\\ 0& {\rm if }\;i=3  \end{cases}, \quad
\psi_{\bar{\eta}_3}  (\varepsilon^3_i) = \begin{cases} -a\varepsilon^2_1, & {\rm if }\;i=0  \\ 0,& {\rm if }\;i=1 \\ 0,& {\rm if }\;i=2 \\  0,& {\rm if }\;i=3  \\  b \varepsilon^2_{3}+ \varepsilon^2_{1}c,& {\rm if }\;i=4 \end{cases}$\\ \hline
& For $\bar{\chi} = \begin{pmatrix} 0 & ab & 0 &  0 \end{pmatrix},$ we get \\ \hline
& $\psi_{\bar{\chi}_1} (\varepsilon^1_i) = 0,\;i=0,1,2,\;$\\
& $\psi_{\bar{\chi}_2} (\varepsilon^2_i) = \begin{cases}0 & {\rm if }\;i=0 \\ a\varepsilon^1_1+\varepsilon^1_0 b, & {\rm if }\;i=1 \\ 0 & {\rm if }\;i=2\\ 0& {\rm if }\;i=3  \end{cases}, \quad
\psi_{\bar{\chi}_3}  (\varepsilon^3_i) = \begin{cases} 0, & {\rm if }\;i=0  \\  a\varepsilon^2_{1}- 2\varepsilon^2_{0}b,& {\rm if }\;i=1 \\ \varepsilon^2_1b,& {\rm if }\;i=2 \\  0,& {\rm if }\;i=3  \\ 0 ,& {\rm if }\;i=4 \end{cases}.$\\ \hline
& For $\sigma = \begin{pmatrix} 0 & 0 & e_1 &  0 \end{pmatrix},$ we get \\ \hline
& $\psi_{\sigma_1} (\varepsilon^1_i) = 0,\;i=0,1,2,$\\
& $\psi_{\sigma_2} (\varepsilon^2_i) =  0,\;i=0,1,2,3$\\
& $\psi_{\sigma_3}  (\varepsilon^3_i) =  0,\;i=0,1,2,3,4$ \\ \hline
\end{tabular}
\caption{Homotopy lifting maps associated to some cocycles in degrees 1,2,3}
\label{table 2}
\end{table}

\end{example}

The following Lemma follows immediately.
\begin{lemma}\label{MC-lemma}
Let $A_1 = kQ/I_1$ be a member of the family of quiver algebras where $I_1 = \langle a^2,b^2,ab-ba, ac\rangle$. The Hochschild 2-cocycles 
$\eta = \begin{pmatrix} a & 0 & 0 &  0 \end{pmatrix}$, $\chi = \begin{pmatrix} 0 & 0 & a &  0 \end{pmatrix}$, $\bar{\eta} = \begin{pmatrix} ab & 0 & 0 &  0 \end{pmatrix}$, $\bar{\chi} = \begin{pmatrix} 0 & ab & 0 &  0 \end{pmatrix}$ and $\sigma = \begin{pmatrix} 0 & 0 & e_1 &  0 \end{pmatrix}$ are Maurer-Cartan elements.
\end{lemma}

\begin{proof}
Let $\gamma$ be any of those elements of $\HH^2(A_1)$. We make use of Equation \eqref{MC-equation}. Since they are all cocycles, $d_3^{*}(\gamma) =0$. Also observe that $\gamma\psi_{\gamma_3}  (\varepsilon^3_i) = 0$ for all $\gamma\in \HH^2(A_1)$, therefore $d_3^*(\gamma) +\gamma\psi_\gamma =0$.
\end{proof}

\section{Deformation of algebras using reduction system}\label{deform-g}

Let $A_1=kQ/I_1$ be a member of family of quiver algebras introduced in Section \ref{finding-mc}, we now show using the combinatorial star product of Equation \eqref{comb-star} that $\HH^2(A)$ has 5 elements satisfying the Maurer-Cartan equation.
\begin{example}\label{final-example}
Recall that for $A_1=kQ/I$, $I= \langle a^2,b^2,ab-ba, ac\rangle$. If we use the set $\{(a^2,0),(b^2,0),(ab,ba),(ac,0)\}$ as the reduction system, this system is reduction finite and reduction unique. All the one overlaps given by $S_3$ resolve to $0$ uniquely. The reduction system
$$R = \{(a^2,0),(b^2,0),(ab,ba),(ac,0)\}$$
satisfies the diamond condition $(\diamond)$ where $\varphi(a^2)=0, \varphi(b^2)=0,\varphi(ab)=ba$ and $\varphi(ac)=0$ . The set $S$ and the set $\text{Irr}_S$ of irreducible paths in the algebra are given respectively by $S = \{ a^2, b^2, ab, ac\}$ and $\text{Irr}_S = \{e_1, e_2, a, b, c, ba, bc\} $, so $dim(A_1)=7$. The paths $a^2$ and $ab$ overlap at $a$ so $(aa)(ab)=a^2b\in S_3$. The set of one-overlaps is given as
$$S_3 = \{a^3, b^3, a^2b, ab^2, a^2c\}.$$
Notice that in the quiver $Q$, the path $a^2, b^2, ab \in S$ are all parallel to the irreducible paths $e_1=e, a, b, ba$ and the path $ac\in S$ is parallel to $c$ and $bc$.
Any element $\widetilde{\varphi}:kS\rightarrow A_1\cong k\text{Irr}_S$ viewed as $\widetilde{\varphi}\in\HHom(KS,A_1)\otimes(\tau)$ has the following general form
\begin{align*}
\widetilde{\varphi}(a^2) &= (\lambda_e + \lambda_a a + \lambda_b b +\lambda_{ba} ba)\tau \\
\widetilde{\varphi}(b^2) &=( \mu_e + \mu_a a + \mu_b b +\mu_{ba} ba) \tau\\
\widetilde{\varphi}(ab) &= (\nu_e + \nu_a a + \nu_b b +\nu_{ba} ba)\tau \\
\widetilde{\varphi}(ac) &=( w_c c + w_{bc} bc)\tau
\end{align*}
for scalars $\lambda_e,\lambda_a,\cdots, w_c, w_{bc}\in k$. By \cite[Corollary 7.37]{SZ}, $\widetilde{\varphi}$ is a Maurer-Cartan element if and only if for each $uvw\in S_3$ with $uv, vw\in S$, Equation \eqref{comb-star} holds. That is
$$(\pi(u)\star\pi(v))\star\pi(w) = \pi(u)\star(\pi(v)\star\pi(w))(mod\; \tau^2)$$
since we are considering first order deformations. We now check conditions on the scalars for the associativity of the star product. This product defined for example for $a,b\in A_1$ is given by
$$a\star b = \varphi(ab)+\widetilde{\varphi}(ab)\tau.$$ 
We check for all elements of $S_3$. For instance, the calculations involved in using $a^3$ to check that $a\star (a\star a) = (a\star a)\star a$ are the following.  
\begin{align*}
a\star (a\star a)  &= a\star (\varphi(a^2)+\widetilde{\varphi}(a^2)\tau) = a\star (\lambda_e + \lambda_a a + \lambda_b b +\lambda_{ba} ba)\tau \\
&= (\lambda_e\varphi(a) + \lambda_a \varphi(a^2) + \lambda_b \varphi(ab) +\lambda_{ba} \varphi(aba))\tau \\
&+  [\lambda_e\widetilde{\varphi}(a) + \lambda_a \widetilde{\varphi}(a^2) + \lambda_b \widetilde{\varphi}(ab) +\lambda_{ba} \widetilde{\varphi}(aba)]\tau^2\\
& =  (\lambda_ea + \lambda_b ba )\tau 
\end{align*}
and it is equal to  
\begin{align*}
(a\star a)\star a &= (\varphi(a^2)+\widetilde{\varphi}(a^2)\tau)\star a = (\lambda_e + \lambda_a a + \lambda_b b +\lambda_{ba} ba)\tau \star a\\
&= (\lambda_e\varphi(a) + \lambda_a \varphi(a^2) + \lambda_b \varphi(ba) +\lambda_{ba} \varphi(ba^2))\tau \\
&+  [\lambda_e\widetilde{\varphi}(a) + \lambda_a \widetilde{\varphi}(a^2) + \lambda_b \widetilde{\varphi}(ba) +\lambda_{ba} \widetilde{\varphi}(ba^2))\tau^2\\
& =  (\lambda_ea + \lambda_b ba )\tau.
\end{align*}
This then implies that $\lambda_e=\lambda_e$ and $\lambda_b=\lambda_b$. For $a^2b = a\star (a\star b) = (a\star a)\star b$, we obtain $ a\star (a\star b) = (\nu_ea + \nu_b ba)\tau$ and $ (a\star a)\star b = (\lambda_eb + \lambda_a ba)\tau$, so we get $\nu_e=\lambda_e=0$ and $\nu_b=\lambda_a$. Equivalent calculations for $ ab^2$ and $a^2c$ yield  $\mu_e=\nu_e=0$ and $\mu_b=\nu_a$ and $\lambda_e=\lambda_b=0$. We can now rewrite 
\begin{align*}
\widetilde{\varphi}(a^2) &=  (\lambda_a a  +\lambda_{ba} ba)\tau \\
\widetilde{\varphi}(b^2) &=  (\mu_a a + \mu_b b +\mu_{ba} ba)\tau \\
\widetilde{\varphi}(ab) &= (\mu_b a + \lambda_a b +\nu_{ba} ba)\tau \\
\widetilde{\varphi}(ac) &= (w_c c + w_{bc} bc)\tau 
\end{align*}
so the Maurer-Cartan elements of $\HH^2(A_1)$ are parametrized by 
$$\widetilde{\varphi}= (\lambda_a,\lambda_{ba}, \mu_a, \mu_b, \mu_{ba},\nu_{ba}, w_c, w_{bc})\in k^8.$$ 
Our next goal is to show that three of these parameters can be eliminated by a coboundary so that $\widetilde{\varphi}\in k^5$ and thus $dim(\HH^2(A_1))=5$. Let $\widetilde{\varphi'}$ be defined by
\begin{align*}
\widetilde{\varphi'}(a^2) &=  (\lambda_a' a  +\lambda_{ba}' ba)\tau, &&\widetilde{\varphi'}(b^2)=  (\mu_a' a + \mu_b' b +\mu_{ba}' ba)\tau \\
\widetilde{\varphi'}(ab) &= (\mu_b' a + \lambda_a' b +\nu_{ba}' ba)\tau, &&\widetilde{\varphi'}(ac) = (w_c' c + w_{bc}' bc)\tau 
\end{align*}
From \cite[Corollary 7.44]{SZ}, two cocycles $\widetilde{\varphi}$ and $\widetilde{\varphi}'$ are cohomologous or satisfy $\widetilde{\varphi}-\widetilde{\varphi} = \langle \Theta\rangle,$ $\Theta\in\HHom(kQ_1, k\text{Irr}_S)$ if 
$$T(\varphi(s))+\widetilde{\varphi}'(s) = T(s_1)\star\cdots\star T(s_m) ( mod\;\tau^2)$$
for some $T:k\text{Irr}_S[\tau]/(\tau^2)\rightarrow \text{Irr}_S[\tau]/(\tau^2)$ defined by $T(x)=x+\Theta(x)\tau$ with $s=s_1s_2\cdots s_m$ a path of length $m$. Any $\Theta\in\HHom(kQ_1, k\text{Irr}_S)$ has a general form
\begin{align*}
\Theta(a) &=  \alpha_e + \alpha_a a + \alpha_b b + \alpha_{ba}ba \\
\Theta(b) &=  \beta_e + \beta_a a + \beta_b b + \beta_{ba}ba \\
\Theta(c)  &= \gamma_c c +\gamma_{bc}bc,
\end{align*}
where $(a\star \Theta(a)) = \alpha_e a +\alpha_a\varphi(a^2) + \alpha_b\varphi(ab)+\alpha_{ba}\varphi(aba))+(\alpha_e\widetilde{\varphi}(a)+\alpha_a\widetilde{\varphi}(a^2) + \alpha_b\widetilde{\varphi}(ab) + \alpha_{ba}\widetilde{\varphi}(aba))\tau.$  Whenever $s=a^2$, then $T(\varphi(a^2))+ \widetilde{\varphi}'(a^2) = T(a)\star T(a)$ yields the following:
\begin{equation}\label{coho-check}
T(\varphi(a^2))+ \widetilde{\varphi'}(a^2) = (\lambda_a' a  +\lambda_{ba}' ba)\tau.
\end{equation}
\begin{align*}
& T(a)\star T(a) = (a+\Theta(a)\tau)\star (a+\Theta(a)\tau) \\
&= a\star a + (a\star \Theta(a))\tau + (\Theta(a)\star a)\tau + (\Theta(a)\star\Theta(a))\tau^2\\
&= (\lambda_a a  +\lambda_{ba} ba)\tau + (\alpha_e a + \alpha_b ba)\tau +  (\alpha_e a + \alpha_b ba)\tau + 0\\
&= (\lambda_a a  +\lambda_{ba} ba + 2\alpha_e a + 2\alpha_b ba)\tau
\end{align*}
and comparing with Equation \eqref{coho-check} we arrive at $\lambda_a' -\lambda_a = 2\alpha_e$ and $\lambda_{ba}' -\lambda_{ba}=2\alpha_b$. With other similar equivalent calculations on $s$ being $b^2, ab, ac$, we get
\begin{align*}
(a^2):&\lambda_a'  -\lambda_a = 2\alpha_e  &&\lambda_{ba}' -\lambda_{ba}=2\alpha_b  \\
(b^2):&\mu_a' - \mu_a =0 &&\mu_b' - \mu_b =2\beta_e   &&&\mu_{ba}' - \mu_{ba} = 2\beta_a \\
(ab): &\mu_b' - \mu_b =\beta_e && \lambda_a'  -\lambda_a = \alpha_e &&&\nu_{ba}' - \nu_{ba} = \alpha_a+\beta_b \\   
(ac):&w_c' - w_c =\alpha_e   && w_{bc}' - w_{bc} = \alpha_b
\end{align*}
This implies that three variables in the parametric definition of $\widetilde{\varphi}$ can be eliminated or simply  $\widetilde{\varphi}= (\lambda_a,\lambda_{ba}, \mu_a, \mu_b, \mu_{ba},\nu_{ba}, w_c, w_{bc})\in k^8$ is cohomologous to $\widetilde{\varphi}= (\lambda_a,\lambda_{ba}, \mu_a, \mu_b, 0 ,\nu_{ba}, 0 , 0)\in k^8.$
Therefore
$\widetilde{\varphi}$ is in $k^5$ or equivalently the dimension of $\HH^2(A_1)$ is equal to $5$.

\end{example}

\end{document}